\documentclass[11pt,a4paper]{article}
 \usepackage[latin1]{inputenc}
\usepackage{amsthm,mathrsfs,amsmath,amssymb}
\usepackage{dsfont}
\usepackage{comment}
\usepackage{indentfirst}
\usepackage{cite}
\usepackage[colorlinks=true]{hyperref}
\hypersetup{urlcolor=blue, citecolor=red}
\numberwithin{equation}{section}
\allowdisplaybreaks
\makeatother
\date{}

\newtheorem{theorem}{Theorem}[section]
\newtheorem{lemma}[theorem]{Lemma}

\newtheorem{proposition}[theorem]{Proposition}

\theoremstyle{definition}
\newtheorem{definition}[theorem]{Definition}
\newtheorem{remark}[theorem]{Remark}

\begin{document}
\baselineskip 15pt \setcounter{page}{1}
\title{ Normalized solutions  for the nonlinear Schr\"{o}dinger equation with potential and    combined nonlinearities\footnote{Supported by
National Natural Science Foundation of China (No. 11971393).}}
\author{{Jin-Cai Kang, \ \ Chun-Lei Tang\footnote{Corresponding author. E-mail address: tangcl@swu.edu.cn (C.-L. Tang)}}\\
{\small \emph{School  of  Mathematics  and  Statistics, Southwest University,  Chongqing {\rm400715},}}
{\small \emph{China}}\\}
\date{}
\maketitle

\baselineskip 15pt

{\bf Abstract:}
  In present paper, we study the following nonlinear Schr\"{o}dinger equation with  combined power  nonlinearities
  \begin{align*}
  - \Delta u+V(x)u+\lambda u=|u|^{2^*-2}u+\mu |u|^{q-2}u \quad \quad \text{in} \ \mathbb{ R}^N, \ N\geq 3
    \end{align*}
  having prescribed  mass
 \begin{align*}
   \int_{ \mathbb{ R}^N}u^2dx=a^2,
  \end{align*}
where  $\mu, a>0$, $q\in(2, 2^*)$,   $2^*=\frac{2N}{N-2}$ is the critical Sobolev exponent,  $V$ is  an  external potential vanishing at infinity, and
the parameter  $\lambda\in \mathbb{R}$ appears as a Lagrange multiplier.  Under some mild assumptions on   $V$, for the $L^2$-subcritical perturbation  $q\in(2, 2+\frac{4}{N})$,  we prove  that there exists $a_0>0$ such that   the  normalized solution with negative energy to the above problem with $\mu>0$ can be obtained when $a\in (0, a_0)$; for the $L^2$-critical perturbation  $q=2+\frac{4}{N}$, by  limiting the range of $\mu$, the positive  ground state normalized solution to the above problem for any $a>0$   is also found with the aid of  the Poho\v{z}aev constraint; moreover,  for  the $L^2$-supercritical perturbation  $q\in( 2+\frac{4}{N}, 2^*)$, we get a positive   ground state  normalized solution for the above problem with $a>0$ and $\mu>0$ by using the Poho\v{z}aev constraint.  At the same time, the exponential decay property of the positive  normalized  solution   is established,  which is important for the  instability analysis    of the standing waves. Furthermore, we give a description of the  ground state  set   and obtain the strong instability of the standing waves  for $q\in[2+\frac{4}{N}, 2^*)$. This paper can be regarded as a generalization of Soave [J. Funct. Anal. (2020)] in a sense.\\



\section{Introduction and main results}
\setcounter{section}{1}
In this paper, we study the existence of  standing waves with prescribed mass for the nonlinear Schr\"{o}dinger equation   with combined power nonlinearities
\begin{align}\label{h51}
i \phi_t+\Delta \phi-V(x) \phi+\mu |\phi|^{q-2}\phi+|\phi|^{p-2}\phi=0\quad \quad \text{in}\ \mathbb{R}^N,
\end{align}
where $N\geq 3$, $\phi: \mathbb{R}\times \mathbb{R}^N\rightarrow \mathbb{C} $, $\mu>0$, $2<q<p =2^*=\frac{2N}{N-2}$ and  $V$ is an  external potential.
 The case   $\mu>0$ is the focusing case, and the case $\mu<0$ is referred to    the defocusing case.
 In the past decade, the   nonlinear Schr\"{o}dinger equation with combined nonlinearities have attracted extensive attention,    starting from the fundamental contribution by   Tao et. al  \cite{2007CPDETao}.  After then, many scholars have done a lot of research on such problem, see for example   \cite{2007CPDETao, 2012DIE, 2016JDECheng, 2018JEE, 2017ARMA, 2016RMI, 2020CMA, 2013CMP, 2017CV}.
 To find stationary states, one makes the ansatz $\phi(t, x) =e^{i \lambda t}u(x)$, where    $\lambda\in \mathbb{R}$ and $u :\mathbb{R}^N \rightarrow \mathbb{C}$ is a time-independent function. Hence, by simple calculation one knows that $u$  satisfies  the following equation
\begin{align}\label{jh1}
  - \Delta u+V(x)u+\lambda u=|u|^{p-2}u+\mu |u|^{q-2}u \quad \quad \text{in} \ \mathbb{ R}^N.
  \end{align}
If we  fix $\lambda\in \mathbb{R}$  to find the solutions $u$ of equation \eqref{jh1}, we call equation \eqref{jh1}  the {\emph{fixed frequency problem}}. One can use the variational method to find the critical points of the corresponding energy  functional of equation \eqref{jh1}
$$
J(u)=\frac{1}{2}\int_{\mathbb{R}^N}|\nabla u|^2+(V(x)+\lambda)|u|^2dx-\frac{1}{q} \int_{\mathbb{R}^N}|u|^qdx- \frac{1}{p} \int_{\mathbb{R}^N}|u|^{p}dx,
$$
or some other topological methods, such as fixed point theory, bifurcation or the Lyapunov-Schmidt reduction. The fixed frequency problem has been  extensively  studied over the past few decades, including but not limited to the existence, non-existence, multiplicity and asymptotic behavior of the solutions (e.g., ground state solution, positive solution and sign-changing solution, and so on).
 Alternatively, one can search for solutions to  equation \eqref{jh1} having prescribed mass
   \begin{align}\label{h1j}
 \int_{\mathbb{R}^N} |u|^2dx=a^2.
\end{align}
The solution of  equation \eqref{jh1} satisfies the prescribed mass constraint \eqref{h1j}, which is called the {\emph{fixed mass problem}}, and in much of the literature called  the normalized solution.
A natural approach to obtain normalized   solution of equation \eqref{jh1}    is to find the critical points   of the   corresponding energy functional   under the constraint \eqref{h1j}.
We refer the
cases $2 <q< 2+\frac{4}{N}$, $q = 2+\frac{4}{N}$ and $  2+\frac{4}{N}<q< 2^* $ as $L^2$-subcritical, $L^2$-critical and $L^2$-supercritical, respectively.

Recently, the question of finding normalized solutions is already interesting for scalar equations, which has attracted extensive attention from scholars.
 For the non-potential case, namely, $V(x)=0$, we  consider the following nonlinear Schr\"{o}dinger equation
\begin{align}\label{h37}
  - \Delta u+ \lambda u=|u|^{p-2}u+\mu |u|^{q-2}u \quad \quad \text{in} \ \mathbb{ R}^N
  \end{align}
 satisfying the normalization constraint  $\int_{ \mathbb{ R}^N}u^2dx=a^2$, where $N\geq 1$ and $2<q \leq p\leq 2^*$. If $p=q\neq \overline{q}:=2+\frac{4}{N}$,  one obtained that, for any $a>0$, equation \eqref{h37}
 has the positive normalized solution   only if $\lambda>0$  by scaling.
 Indeed,      from \cite{1989-K},     the following Schr\"{o}dinger equation
\begin{align*}
\begin{cases}
 - \Delta u+  u=|u|^{q-2}u  \quad \quad &\text{in} \ \mathbb{ R}^N,\\
 u(x)\rightarrow 0 &\text{as}\ |x|\rightarrow\infty
  \end{cases}
  \end{align*}
  had the unique positive radial solution  $W_q$.
Letting
$$
W_{\lambda, q}(x):= \lambda^{\frac{1}{q-2}} W_q(\sqrt{\lambda} x),
$$
 it is easy to check that  $W_{\lambda, q}$ is the unique positive radial solution to equation \eqref{h37} up to a translation. By  direct computation, one shows that
 $$
 |W_{\lambda, q}|_2^2=\lambda^{\frac{4-(q-2)N}{2(q-2)}}  |W_{ q}|_2^2.
 $$
Thus, if  $p=q\neq \overline{q}:=2+\frac{4}{N}$, for any $a>0$,  there exists a unique $ \lambda_a>0$ such that  $|W_{\lambda_a, q}|_2^2=a^2$. So for any $a>0$,  $W_{\lambda_a, q}$  is a positive normalized solution  to equation \eqref{h37} satisfying  $|W_{\lambda_a, q}|_2^2=a^2$   whenever $p=q\neq2+\frac{4}{N}$ (and it is unique up to a translation).
  While for the so-called mass critical case $p=q= \overline{q}:=2+\frac{4}{N}$,  equation \eqref{h37}  has a positive normalized solution if and only if $a=|W_{\lambda, \overline{q}}|_2 $.
  However,  when $q\neq p$,    the scaling method does not work.  For the case    $2<q<p<2+\frac{4}{N}$ with $N\geq 1$   and the case $2+\frac{4}{N}<q<p<2^*$ with $2^*=\frac{2N}{N-2}$ if $N\geq 3$ and $2^*=+\infty$ if $N=1,2$,   the corresponding  energy functional to equation \eqref{h37} is bounded below and unbounded below on $S_a$, respectively, where
    \begin{align*}
S_a=\Big\{u\in H^1(\mathbb{R}^N): \int_{\mathbb{R}^N} |u|^2dx=a^2 \Big\}.
\end{align*}
    Naturally, the techniques for dealing with these  cases are also different and the results for the existence of  normalized   solutions to these  cases can be  found in \cite{S1980, S1981, S1989, Shi2014, 1997jean} and references therein.
Next, when it comes to combined nonlinearities, the following works deserve  to be highlighted.
Soave in \cite{2020Soave} first studied  the existence and nonexistence of the normalized solution for equation \eqref{h37} with $\mu\in \mathbb{R}$  and   combined power nonlinearities  $2<q\leq 2+\frac{4}{N}\leq p<2^*$ where $N\geq1$ and made pioneering work by using  variational method and Poho\v{z}aev constraint. In particular, when $2<q< 2+\frac{4}{N}< p<2^*$, he obtained the existence of  two solution (local minimizer and Mountain-Pass type) for  equation \eqref{h37}. Moreover, he    got the orbital stability of the ground state set when $2<q< 2+\frac{4}{N}< p<2^*$  and the strongly  instability of stand wave when $q= 2+\frac{4}{N}< p<2^*$.
After then, he  in \cite{2021-JFA-Soave}   further considered   the existence and nonexistence of the  normalized  solution for equation \eqref{h37}  with $\mu\in \mathbb{R}$, $p=2^*=\frac{2N}{N-2}$  and $ q\in(2,2^*)$ where $N\geq 3$ by applying
 a similar technique in   \cite{2020Soave}. However, in the case of $2<q< 2+\frac{4}{N}< p=2^*$, he obtained only the existence of local minimizer for equation \eqref{h37}. About the second  normalized   solution
 (Mountain-Pass type)  for  equation \eqref{h37} with  $2<q< 2+\frac{4}{N}< p=2^*$,  it was given by  Jeanjean et  al. \cite{2022-MA-jean} for $N\geq 4$ and Wei et  al. \cite{2022JFAWei} for $N= 3$,  respectively.
 And then, Chen et  al. \cite{2022.9} proposed new strategies to control the energy level in the Sobolev critical case which allow to treat, in a unified way, the dimensions $N = 3$ and $N \geq4$, and obtained  the second solution (Mountain-Pass type) for  equation \eqref{h37}   with $2<q< 2+\frac{4}{N}< p=2^*$.
 Besides, the generalizations and improvements for  \cite{2021-JFA-Soave}   was carried out in
the  recent paper \cite{2021CVLixinfu, 2022CValves  }. To be specific,  Li \cite{2021CVLixinfu}  removed the   restriction on $\mu$ and got the ground state normalized solution for  equation \eqref{h37} when $ 2+\frac{4}{N}<q< p=2^*$. And then, Alves et  al. \cite{2022CValves }   considered the existence of normalized solution with exponential critical growth   for $N=2$.

 Now, a lot of scholars have focused on   the following mass prescribed problem with potential
  \begin{align}\label{h38}
  - \Delta u+V(x)u+\lambda u= g(u) \quad \quad \text{in} \ \mathbb{ R}^N
  \end{align}
  satisfying the normalization constraint  $\int_{ \mathbb{ R}^N}u^2dx=a^2$, where $N\geq1$ and  $V$ is an external potential.
Note that  the techniques  used  in the literature mentioned above studying the non-potential case  can not be applied directly.  Therefore, such problem are also challenging and stimulating for researchers.
When dealing with the mass sub-critical case involving potential and the general nonlinearities, the functional is bounded from below, thus one can apply the minimizing argument constraint on $S_a$.  The main difficulty is the compactness of the minimizing sequence. In order to solve this problem, it is important to get the strict so-called sub-additive inequality. We mentioned that Ikoma et al.   \cite{2020CVIkoma} first made progress in this direction by applying the  standard concentration compactness arguments, which is due to Lions \cite{1984Lions1,1984Lions2}.
  Whereafter,   to obtain the strict sub-additive inequality,  a new approach  was   proposed    by   Zhong and Zou \cite{2021-Zhong} based on iteration.
 For more existence results can be obtained in  \cite{2022JGAYang,2022JGAAlves, Pe }  and references therein under some different  assumptions of $g$ and $V$.

When we consider the mass super-critical case involving potential, there
are few works   on equation \eqref{h38}.
 More precisely,
when $g(u)=|u|^{q-2}u $ and $2+ \frac{4}{N} < q <2^*$,  the authors in  \cite{2021CPDEB} considered the existence of normalized solutions  with high Morse index for equation \eqref{h38} with the   positive potential $V(x)\geq0$  vanishing at infinity   by constructing a suitable linking structure, because the
mountain pass structure in Jeanjean \cite{1997jean}  does not work in this case.    Soon after, when   $g(u)=|u|^{q-2}u $ with $2+ \frac{4}{N} < q <2^*$ and $V(x)  \leq0$   vanishing at infinity,   the authors in \cite{Molle}   proved the existence of    normalized  solutions  for equation \eqref{h38}     under some explicit smallness assumption on $V(x)$.  Afterwards,   Ding and Zhong \cite{2022JDEDing} considered the   general nonlinearities $g$   and   got the existence of ground state normalized solutions   to equation \eqref{h38} with negative potential  by applying the Poho\v{z}aev constraint  under an explicit smallness assumption on $V$   and some Ambrosetti-Rabinowitz type conditions on $g$.
  For more results about
 normalized solutions to equation \eqref{h38}  with $V(x)=0$ or  $V(x)\neq0$, we refer readers to \cite{2019-APDE, 2021JFA-B, 2022-N-Tang, 2022-SCM, 2020-AM-Yang}  and references therein.

Motivated by \cite{2021-JFA-Soave, 2022-MA-jean, 2022JDEDing }, a natural question arises as to whether  the normalized solution exists if  $V(x)\neq 0$  vanishing at infinity and $g$ is  combined nonlinearities
 in equation \eqref{h38}  satisfying the normalization constraint  $\int_{ \mathbb{ R}^N}u^2dx=a^2>0$. Hence,
we investigate the following nonlinear Schr\"{o}dinger equation with       potential   and combined power nonlinearities  in this paper
 \begin{align}\label{h1}\tag{${ \mathcal{P}}$}
  \begin{cases}
  - \Delta u+V(x)u+\lambda u=|u|^{2^*-2}u+\mu |u|^{q-2}u \quad \quad \text{in} \ \mathbb{ R}^N,\\
   \int_{ \mathbb{ R}^N}u^2dx=a^2,
 \end{cases}
  \end{align}
 where $N\geq 3$,  $\mu>0$,   $q\in(2, 2^*)$,   $2^*=\frac{2N}{N-2}$ is the critical Sobolev exponent, $a>0$ is prescribed, and the parameter  $\lambda\in \mathbb{R}$
 appears as a Lagrange multiplier. To
reach the conclusions of this paper, we give the following assumptions.
 If $q\in(2+\frac{4}{N}, 2^*)$,  we assume that $V$ satisfies
 \begin{itemize}
   \item [($V_1$)] Let   $V\in C^2(\mathbb{R}^N, \mathbb{R})$ and
    $\lim_{|x|\rightarrow\infty}V(x)=\sup_{x\in \mathbb{R}^N}V(x)=0$.  Moreover,   there exists $0<\sigma_1< \min\big\{\frac{1}{2}, 1-\frac{4}{N(q-2)} \big\} $ such that
       $$
       \Big| \int_{ \mathbb{ R}^N}V(x)|u|^2dx\Big|\leq \sigma_1 |\nabla u|_2^2  \quad \quad \text{for all} \ u\in H^1(\mathbb{R}^N, \mathbb{C});
       $$
   \item [($V_2$)] let $W(x)=\frac{1}{2}\langle \nabla V(x),x\rangle$ and $\lim_{|x|\rightarrow\infty}W(x)=0$.  There is $0<\sigma_2 <\min \big\{  \frac{2}{N-2}- \frac{N}{N-2}\sigma_1, 1-\frac{4}{(q-2)N}, \frac{N(q-2) }{4} (1-\sigma_1)-1 \big \}  $ such that
       $$
       \Big| \int_{ \mathbb{ R}^N}W(x)|u|^2dx\Big|\leq \sigma_2 |\nabla u|_2^2  \quad \quad \text{for all}  \ u\in H^1(\mathbb{R}^N, \mathbb{C});
       $$
   \item [($V_3$)]  let $L(x)= \langle\nabla W(x),x \rangle$. There exists $0<\sigma_3 <   \frac{q-2}{2}N(1- \sigma_2 )-2$ such that
       $$
       \Big| \int_{ \mathbb{ R}^N}L(x)|u|^2dx\Big|\leq \sigma_3 |\nabla u|_2^2  \quad \quad \text{for all}  \ u\in H^1(\mathbb{R}^N, \mathbb{C});
       $$
   \item [($V_4$)] $V(x)+W(x)\leq 0$ a.e.  on $\mathbb{ R}^N$.
      \end{itemize}

 If $q=\overline{q}:=2+\frac{4}{N} $,   we give the following hypothesises for $V$.
      \begin{itemize}

      \item [($\widetilde{V}_1$)] $V\in C^2(\mathbb{R}^N, \mathbb{R})$ and $\lim_{|x|\rightarrow\infty}V(x)=\sup_{x\in \mathbb{R}^N}V(x)=0$, and there exists $0<\widetilde{\sigma}_1<\min\big\{ \frac{2}{N}, \frac{1}{2}\big \}  $ such that
       $$
       \Big| \int_{ \mathbb{ R}^N}V(x)|u|^2dx\Big|\leq \widetilde{\sigma}_1 |\nabla u|_2^2 \quad \quad  \text{for all}  \ u\in H^1(\mathbb{R}^N, \mathbb{C});
       $$
    \item [($\widetilde{V}_2$)] let $W(x)=\frac{1}{2}\langle \nabla V(x),x\rangle$, $\lim_{|x|\rightarrow\infty}W(x) =0$. There exists $0<\widetilde{\sigma}_2 <\min\big\{ \frac{2}{N-2}-\frac{N}{N-2} \widetilde{\sigma}_1, \frac{2}{N}   \big  \}  $ such that
       $$
       \Big| \int_{ \mathbb{ R}^N}W(x)|u|^2dx\Big|\leq \widetilde{\sigma}_2 |\nabla u|_2^2  \quad \quad   \text{for all}  \ u\in H^1(\mathbb{R}^N, \mathbb{C});
       $$
        \item [($\widetilde{V}_3$)]  let $L(x)= \langle\nabla W(x),x \rangle$. There exists $0<\widetilde{\sigma}_3 <  \frac{4}{N-2} -\frac{2N}{N-2}\widetilde{\sigma}_2$ such that
       $$
       \Big| \int_{ \mathbb{ R}^N}L(x)|u|^2dx\Big|\leq \widetilde{\sigma}_3 |\nabla u|_2^2  \quad \quad   \text{for all}  \ u\in H^1(\mathbb{R}^N, \mathbb{C}).
       $$
 \end{itemize}

Before presenting our main results, we give some notations that will be used throughout the paper.
The notation $\mathcal{D}^{1,2}(\mathbb{R}^N, \mathbb{C})$ (or $\mathcal{D}^{1,2}(\mathbb{R}^N, \mathbb{R})$) is the usual Sobolev space with the norm $$
\|u\|_{\mathcal{D}^{1,2}    }= \int_{\mathbb{R}^N} |\nabla u|^2 dx.
$$
$H^1(\mathbb{R}^N, \mathbb{C}) $ (or $H^1(\mathbb{R}^N, \mathbb{R}) $) denotes the usual Sobolev space with the norm
$$
\|u\|_{H }= \int_{\mathbb{R}^N} |\nabla u|^2+  |u|^2dx.
$$
$L^p(\mathbb{R}^N, \mathbb{C} )$ (or  $L^p(\mathbb{R}^N, \mathbb{R} )$ ) is the Lebesgue space endowed with the norm
$$
|u|_{L^p(\mathbb{R}^N)}=|u|_p=\left(\int_{\mathbb{R}^N}|u|^pdx\right)^{\frac{1}{p}}, \quad\quad \text{ where}  \ p\in[1,+\infty),
$$
 and  $|u|_{L^{\infty}(\mathbb{R}^N)}=|u|_{\infty}=  \text{ess} \sup_{x\in\mathbb{R}^N }|u(x)|$.   $C$ and $C_i$, $i=1,2, \ldots$, denote positive constants possibly different in different places. $o(1)$ denotes a quantity which goes to zero as $n\rightarrow\infty$.  Moreover, we
denote by $\mathbb{R}^+$ the interval $(0, +\infty)$.
     Now, we introduce the work space in this paper
\begin{align*}
E=\Big\{u\in H^1(\mathbb{R}^N, \mathbb{C}): \int_{\mathbb{R}^N} V(x)|u|^2dx<+\infty \Big\}
\end{align*}
endowed with the inner product and the induced norm
$$
\langle u,\varphi\rangle=\int_{\mathbb{R}^N} \nabla u\cdot \nabla \overline{\varphi}+u\overline{\varphi}+V(x)u\overline{\varphi} dx\quad \quad \text{and}\quad \quad \|u\|=\Big(\int_{\mathbb{R}^N} |\nabla u|^2+|u|^2+V(x)|u|^2dx\Big)^\frac{1}{2}£¬
$$
where  $\overline{\varphi}$ is the complex conjugate of $\varphi$.
 Moreover, under our assumptions of  $V$ in present paper, we deduce  that the  norm 	$\|u\|$  is equivalent to the usual norm 	$\|u\|_{H }$.	 The embedding $E\hookrightarrow{L^p{(\mathbb{R}^N, \mathbb{C})}}$ is continuous  with  $p\in[2, 2^*)$  and $N\geq 3$.
Since the energy functional $ \mathcal{I}_{q}:E\rightarrow \mathbb{R}$    of Eq. \eqref{h1} defined by
  \begin{align}\label{h2}
\mathcal{I}_{q}(u)=\frac{1}{2}\int_{\mathbb{R}^N}|\nabla u|^2+V(x)|u|^2dx
-\frac{1}{2^*}\int_{\mathbb{R}^N}|u|^{2^*}dx
-\frac{\mu}{q}\int_{\mathbb{R}^N}|u|^{q}dx
\end{align}
is unbounded
from below constraint on the $L^2$-sphere $S_a$, where
 \begin{align*}
S_a=\Big\{u\in E: \int_{\mathbb{R}^N} |u|^2dx=a^2   \Big\},
\end{align*}
 thus one can not apply the minimizing argument constraint on   $S_a$ any more.
 Therefore, we intend to  consider the minimization of $\mathcal{I}_{q}$ on a subsets of $S_a$ in this work.
It is easy to see  that if $u$ is a solution to  Eq. \eqref{h1}, then  the following Nehari  identity holds
 \begin{align}\label{hjh4}
\int_{\mathbb{R}^N}|\nabla u|^2+V(x)|u|^2dx+\lambda\int_{\mathbb{R}^N}  |u|^2dx  - \int_{\mathbb{R}^N}|u|^{2^*}dx-\mu  \int_{\mathbb{R}^N}|u|^{q}dx=0.
\end{align}
Moreover, from  \cite[Proposition 2.1]{2014-JFA-le} (see also  \cite[Lemma 3.1]{2022JDEDing}) one obtains that if $u\in E$ is a solution to  Eq. \eqref{h1}, $u$ satisfies the following Poho\v{z}aev identity
 \begin{align}\label{hhjh4}
\frac{N-2}{2}\int_{\mathbb{R}^N}|\nabla u|^2dx +\frac{N}{2}\int_{\mathbb{R}^N}V(x)|u|^2dx&+ \frac{1}{2}\int_{\mathbb{R}^N}\langle\nabla V(x), x\rangle |u|^2dx
+\frac{\lambda N}{2}\int_{\mathbb{R}^N}  |u|^2dx\nonumber\\
 & -  \frac{N}{2^*}\int_{\mathbb{R}^N}|u|^{2^*}dx-\frac{\mu N}{q} \int_{\mathbb{R}^N}|u|^{q}dx=0.
\end{align}
Hence, by \eqref{hjh4} and \eqref{hhjh4},  $u$ satisfies
 \begin{align}\label{h4}
P_q(u)=0,
\end{align}
where
$$
P_q(u):=\int_{\mathbb{R}^N}|\nabla u|^2dx-\frac{1}{2}\int_{\mathbb{R}^N} \langle \nabla V(x), x\rangle |u|^2dx- \int_{\mathbb{R}^N}|u|^{2^*}dx-\mu \gamma_q \int_{\mathbb{R}^N}|u|^{q}dx
$$
with $ \gamma_q=\frac{N(q-2)}{2q}=\frac{N}{2}-\frac{N}{q} $.  \eqref{hhjh4} and \eqref{h4} are also  called Poho\v{z}aev identity. In particular, \eqref{h4} is widely used in the
literature when study the prescribed mass problem. Thus,
 if not otherwise specified, the Poho\v{z}aev identity we used   in this paper  is   \eqref{h4}.
   To find the  normalized    solutions of Eq.  \eqref{h1}, we introduce the following  Poho\v{z}aev constrained
set
$$
\mathcal{P}_{q,a}=\{u\in E: P_q(u)=0\} \cap S_a.
$$
Obviously, the set $\mathcal{P}_{q,a}$ contains all of the   normalized    solutions to Eq.  \eqref{h1}.
Hence,   this paper mainly studies the following minimization problem
$$
m_{q,a}=\inf_{u\in\mathcal{P}_{q,a}} \mathcal{I}_q(u).
$$
For the limit problem to Eq. \eqref{h1},  namely,
 \begin{align}\label{h1jx}\tag{${ \mathcal{P}}_\infty$}
  \begin{cases}
  - \Delta u+ u+\lambda u=|u|^{2^*-2}u+\mu |u|^{q-2}u \quad \quad \text{in} \ \mathbb{ R}^N,\\
   \int_{ \mathbb{ R}^N}u^2dx=a^2>0,
 \end{cases}
  \end{align}
  where $N\geq 3$, $\mu>0$, $q\in(2, 2^*)$ with  $2^*=\frac{2N}{N-2}$ and the parameter  $\lambda\in \mathbb{R}$
 appears as a Lagrange multiplier.
 Defined the energy functional  $\mathcal{I}_{q}^\infty: H^1(\mathbb{R}^N, \mathbb{C})\rightarrow \mathbb{R}$ of Eq. \eqref{h1jx}
 \begin{align}\label{h3}
\mathcal{I}_{q}^\infty(u)=\frac{1}{2}\int_{\mathbb{R}^N}|\nabla u|^2dx
-\frac{1}{2^*}\int_{\mathbb{R}^N}|u|^{2^*}dx
-\frac{\mu}{q}\int_{\mathbb{R}^N}|u|^{q}dx,
\end{align}
and the constrained set and the least energy of the normalized    solutions for
Eq.  \eqref{h1jx} is given  by
$$
\mathcal{P}_{q,a}^\infty=\{u\in H^1(\mathbb{R}^N, \mathbb{C}): P_q^\infty(u)=0\}  \cap S_a   \quad \quad \text{and}\quad \quad m_{q,a}^ \infty=\inf_{u\in\mathcal{P}_{q,a}^\infty } \mathcal{I}_q^\infty(u),
 $$
where
\begin{align}\label{h5j}
P_q^\infty(u)=\int_{\mathbb{R}^N}|\nabla u|^2dx - \int_{\mathbb{R}^N}|u|^{2^*}dx-\mu \gamma_q \int_{\mathbb{R}^N}|u|^{q}dx
\end{align}
with $ \gamma_q=\frac{N(q-2)}{2q}=\frac{N}{2}-\frac{N}{q} $.\\

We will be particularly interested in ground state solutions, defined as
follows:
  \begin{definition}\label{de1.1}
  We say that  $\widetilde{u}$ is a ground state of  Eq. \eqref{h1}  on $S_a$ if it is a solution to Eq. \eqref{h1} having minimal energy among all the solutions which belongs to $S_a$:
\begin{align}
  d\mathcal{I}_q|_{S_a }(\widetilde{u})=0 \quad \text{and} \quad \mathcal{I}_q(\widetilde{u})=\inf\{\mathcal{I}_q(u):   d\mathcal{I}_q|_{S_a }(u)=0 \ \text{and}\ u\in S_a \}.
\end{align}
The set of the ground states will be denoted by $Z_a$.
 \end{definition}

We also recall the notion of   instability  that we are interested in:

  \begin{definition}\label{de1.2}
 A standing wave $e^{i\lambda t}u$ is strongly unstable if for every $\varepsilon> 0$, there exists $\varphi_0 \in H^1(\mathbb{R}^N , \mathbb{C})$ such
that $\|u- \varphi_0\|_{H}<\varepsilon$,
    and $\varphi(t,\cdot)$ blows-up in finite time, where $\varphi(t,\cdot)$ denotes the solution to \eqref{h51} with initial datum  $\varphi_0$.
 \end{definition}

\vspace{0.2cm}

Now, we give the main results in this paper.

\begin{theorem}\label{TH1}
 Let $N \geq3$. Assume that  $q\in (2+\frac{4}{N}, 2^*)$ and $(V_1)-(V_4)$ hold.   Then for any $\mu>0$ and $a>0$, there exists a couple  $(\lambda, u) \in \mathbb{R}^+ \times E$ solving Eq. \eqref{h1}, where $u$ is a   positive real-valued function in $\mathbb{R}^N$ and
 $ \mathcal{I}_q(u)=m_{q,a}$.
\end{theorem}

\begin{theorem}\label{TH2}
 Let $N \geq3$ and $a>0$. Assume that  $q=\overline{q}:=2+\frac{4}{N}$ and $(\widetilde{V}_1)-(\widetilde{V}_3)$ and $(V_4)$ hold. If
$$0<\mu a^{\frac{4}{N}}< \min\left\{1-\widetilde{\sigma}_2, \frac{N+2}{N}- \big(2+\frac{4}{N}\big)\widetilde{\sigma}_1,    1 -\frac{N \widetilde{\sigma}_1}{2}-\frac{N -2 }{2} \widetilde{\sigma}_2,
 \frac{2}{N}-\widetilde{\sigma}_2-\frac{\widetilde{\sigma}_3}{2^*}     \right\}(\overline{a}_N)^{\frac{4}{N}},
 $$
 where $ \overline{a}_N$ is defined in \eqref{h46},  then  there exists a couple  $(\lambda, u) \in \mathbb{R}^+ \times E$ that solves Eq. \eqref{h1},
  where $u$ is a positive real-valued  function in $\mathbb{R}^N$ and $
  \mathcal{I}_{\overline{q}}(u)=m_{\overline{q},a}$.
 \end{theorem}

 \begin{theorem}\label{TH6}
 Let $N \geq3$. Assume that  $2<q<\overline{q}:=2+\frac{4}{N}$ and  $V$ satisfies $(V_4)$ and
 \begin{itemize}
   \item [$(\widehat{V}_1)$]    Let   $V\in C^1(\mathbb{R}^N, \mathbb{R})$ and
    $\lim_{|x|\rightarrow\infty}V(x)=\sup_{x\in \mathbb{R}^N}V(x)=0$.  Moreover,   there exists $0<\widehat{\sigma}_1< 1  $ such that
       $$
       \Big| \int_{ \mathbb{ R}^N}V(x)|u|^2dx\Big|\leq \widehat{\sigma}_1 |\nabla u|_2^2  \quad \quad   \text{for all}  \ u\in H^1(\mathbb{R}^N, \mathbb{C}).
       $$
 \end{itemize}
Then, for any $\mu>0$, there exists $a_0=a_0(\mu)>0$ such that  Eq. \eqref{h1} has a solution  $(\lambda, u) \in \mathbb{R}^+ \times E$
 for any $a\in(0, a_0)$.
 \end{theorem}

\begin{remark}
To our best knowledge, it seems to be the first work on the existence of normalized  solution for the Schr\"{o}dinger equation with potential
and  combined power nonlinearities.
The appearance of the potential term will affect the geometry of the problem. More specifically, when $q \in [\overline{q}, 2^*)$, verifying the geometry of the problem becomes more complicated with respect to the case   $V(x)=0$  (see Lemmas \ref{Lem2.3} and \ref{Lem2.3j}). And when $q \in (2, \overline{q})$,    it is difficult to get the   geometry for the problem with respect to the case   $V(x)=0$   as in \cite[Lemma 4.2]{2021-JFA-Soave}.
\end{remark}

\begin{remark} Because of   the potential   and  the Sobolev critical term,
  the main difficulty we encounter in proving  the existence results is the lack of compactness. Compared with \cite{2021-JFA-Soave}, since
 the embedding $E\hookrightarrow{L^p{(\mathbb{R}^N, \mathbb{C})}}$ is   not compact with  $p\in(2, 2^*)$  and $N\geq 3$ in this paper,
  we can not solve it by using a similar technique as in \cite{2021-JFA-Soave}.
It must be stressed that  Proposition  \ref{TH3} plays a key role in overcoming this difficulty and verifying Theorems \ref{TH1} and \ref{TH2}.
Furthermore, motivated by \cite{2022-JMPA-jean}, we obtain a solution with negative energy for Eq. \eqref{h1} satisfying $|u|_2=a\in(0, a_0)$ when $q\in (2, \overline{q})$. Since the translation invariance properties of the problem  established in \cite{2022-JMPA-jean} is not valid in  this paper, we cannot apply the methods in \cite{2022-JMPA-jean} to exclude the vanishing for the weak limit    of the minimizing sequences on a related problem. To achieve it, we make full use of Lemma \ref{Lem5.5}. Then we show that a minimizing sequence with nontrivial weak limit of  $\mathcal{I}_q$ is precompact with the aid of Lemma \ref{Lem5.4} and prove Theorem \ref{TH6}. However,    we do  not know whether  the normalized   solution obtained in Theorem \ref{TH6} is the     ground state normalized solution under the sense of Definition \ref{de1.1}.
\end{remark}

\begin{remark}
Using a similar discussions and techniques   as above,  we  can  obtain the existence of normalized solutions for the  following problem
 \begin{align}\label{h86}
  \begin{cases}
  - \Delta u+V(x)u+\lambda u=|u|^{p-2}u+\mu |u|^{q-2}u \quad \quad \text{in} \ \mathbb{ R}^N,\\
   \int_{ \mathbb{ R}^N}u^2dx=a^2>0,
 \end{cases}
  \end{align}
where $N\geq 3$, $\mu>0$,  $2<q<p<2^*$,  and the parameter  $\lambda\in \mathbb{R}$  appears as a Lagrange multiplier.  To be more precise, the method  of Theorem  \ref{TH1} and  Theorem  \ref{TH2} can be used in the case  of $2+\frac{4}{N}\leq q<p<2^*$;  when $ 2<q<2+\frac{4}{N} <p<2^*$, we can refer to the proof of Theorem \ref{TH6}.
\end{remark}

\begin{remark}
Unlike the previous paper \cite{2022-JMPA-jean, 2020Soave, 2022JFAWei, 2022.9},  it seems that it is not easy to get the second solution (Mountain-Pass type)  on the case of $2<q<2+\frac{4}{N}<p\leq 2^*$ for problem \eqref{h86} in this paper due to the lack of compactness.
 Therefore,  finding the second solution (Mountain-Pass type) for  problem   \eqref{h86}  with $2<q<2+\frac{4}{N}<p\leq 2^*$ is also an interesting problem.
\end{remark}

In what follows, we give the exponential decay property of the positive solution     for Eq. \eqref{h1}.

\begin{theorem}\label{JTH2}
Suppose that   $(V_1)$ and  $(V_4)$ or $(\widetilde{V}_1)$ and  $(V_4)$   hold. Let $u\in E$ be the positive real-valued solution
for Eq. \eqref{h1} with  $\mu>0$, $a>0$ and $q\in[2+\frac{4}{N}, 2^*)$.
 Then the corresponding  Lagrange multiplier $\lambda>0$, and there exists a constant $M>0$   such that
$$
  | u(x)| \leq   Me^{\big(-\sqrt{\frac{1}{2}}|x|\big)}\quad\quad \text{for all}\ x\in \mathbb{R}^N.
    $$
\end{theorem}

\begin{remark}  It must be emphasized that the exponential decay     of the positive solution for Eq. \eqref{h1}     plays an important role in proving the strong instability of the stand wave   for problem \eqref{h51}.
In \cite{2021-JFA-Soave},  it is known from Radial Lemma \cite{1983-BL}  that the radial solution obtained by  the author  naturally satisfies the decay property $|u(x)|\rightarrow 0$ as $|x|\rightarrow\infty$. However, unlike quoted paper \cite{2021-JFA-Soave},
   the Radial Lemma \cite{1983-BL} cannot be used  in our case.     So the difficulty we face is that it is hard to verify $|u(x)|\rightarrow 0$ as $|x|\rightarrow\infty$.
     To overcome this obstacle,  we will use a new technique  in this work, i.e., Moser iteration technique, which is mainly inspired by  \cite{2001book, 2013cv-Byeon}.
\end{remark}

\vspace{0.2cm}

When $V(x)=0$, the following issue was  firstly studied by Tao et. al \cite{2007CPDETao},  where  they proved   the occurrence of finite-time blow-up   for  focusing $L^2$-supercritical  and Sobolev critical case.  After then, Soave in \cite{2021-JFA-Soave}  gave a similar results  in a different (and complementary) perspective. Inspired by the above   studies,  we have the following result.

\begin{theorem}\label{TH4}
Suppose that
 \begin{itemize}
   \item [$(V_5)$]  there exist  $\epsilon_1, \epsilon_2 >0$ small enough  such that $ |V|_{ \frac{N^2}{2N-4}}<\epsilon_1 $ and $|\nabla V|_{ \frac{N^2}{2N-4}}<\epsilon_2$.
 \end{itemize}

  (i) Under the assumptions of Theorem \ref{TH1} or  Theorem \ref{TH2},
   let $u_0 \in S_a$ be such that $\mathcal{I}_q(u_0)<\inf_{u\in \mathcal{P}_{q,a}}\mathcal{I}_q(u)$,
 and if  $|x|u_0 \in L^2(\mathbb{R}^N, \mathbb{C})$ and $s_{u_0}<0$, where $s_{u_0}$ is defined in Lemma \ref{Lem2.3} or  Lemma \ref{Lem2.3j}, then the solution $\phi$ of problem \eqref{h51}  with initial datum $u_0$ blows-up
in finite time.

(ii) In particular,   let $q\in(\overline{q}, 2^*)$ and let $(V_1)-(V_2)$ be hold.   If we   assume that  $\mathcal{I}_q(u_0)<0$,   $|x|u_0 \in L^2(\mathbb{R}^N, \mathbb{C})$ and $y(0)=y_0>0$  defined in \eqref{h69},  then    the blow-up time $T$ can be estimate by
$$
0<T\leq \frac{|xu_0|_2^2}{( q\gamma_q-2-\frac{1}{2}\sigma_2-q\gamma_q \sigma_1 ) y_0}.
$$
\end{theorem}

\begin{remark}
  The assumption $(V_5)$  is only used to ensure that the  initial-value problem \eqref{h51} with initial datum   is local well-posed on $(-T_{min}, T_{max})$ with $T_{min}, T_{max}>0 $. More precisely, as in the quoted paper \cite{2022-JMPA-jean},    our
purpose is to show that the Duhamel functional
$$
\Theta (u)(t):=e^{it \Delta}u(0)+ i\int_0^t e^{i(t-s) \Delta} (|u|^{2^*-2}u+|u|^{q-2}u-V(x)u)ds
$$
is a contraction on   a particular complete metric space, e.g., the metric space $(B_{2\gamma_0, T}, d)$ in \cite{2022-JMPA-jean}.  If $V\in L^\infty(\mathbb{R}^N)$
 is sufficient to get the result we want, then the assumption $(V_5)$ can be removed. Moreover,  there are many functions that satisfy  the assumptions  $(V_1)-(V_5)$ or  $(\widetilde{V}_1)-(\widetilde{V}_3)$  or $(\widehat{V}_1)$  in this paper,  for example,  $V(x)=- \frac{\kappa}{|x|^2+\ln(|x|+2)}$ with $\kappa>0$ small enough.
\end{remark}

\vspace{0.2cm}
Based on  the results mentioned above,
  we further   study the strong instability  of the stand waves   for problem \eqref{h51}.


\begin{theorem}\label{TH5} 
Assume that $(V_5)$ holds.
Under the assumptions of Theorem \ref{TH1} or Theorem \ref{TH2},
   $Z_a$ has the following characteristics
 \begin{align*}
 Z_a=\{ e^{i \theta}u:  \ \theta\in \mathbb{R}, \ u\in  U_a\ \text{and}\  u>0 \ \text{in}\ \mathbb{R}^N \},
 \end{align*}
 where $U_a$ is defined in \eqref{h91}.
  Moreover, if $\widehat{u}$ is a ground state, then the associated Lagrange multiplier $\widehat{\lambda}$ is positive, and the
standing wave $e^{i\lambda t}\widehat{u} $ is strongly unstable.
\end{theorem}

\begin{remark}
  To prove it, we  will make use of  Theorems \ref{JTH2} and  \ref{TH4}.
 Compared to the previous literature \cite{2020Soave, 2021-JFA-Soave},  which strongly relies on the solution of Eq. \eqref{h1}  being a radial function, the novelty in this paper is that  the solution of Eq. \eqref{h1}  is no longer restricted to a radial function.
\end{remark}

We conclude this section by giving    the organization of this paper. In Sec. \ref{sec2},  we give some preliminaries, which are
important to justify our results.       Then we will complete the proof of  Theorem   \ref{TH1}   in Sec. \ref{sec3}.   Whereafter, we discuss the $L^2$-critical case  and prove Theorem  \ref{TH2}  in  Sec. \ref{sec4}. The proof of Theorem   \ref{TH6} is given in Sec. \ref{sec6}.  In Sec. \ref{sec5},  we consider the  strong instability of the standing wave for problem  \eqref{h51} and prove Theorems   \ref{JTH2}, \ref{TH4} and \ref{TH5}. Throughout this paper, we assume that $\mu>0$ and $a>0$ in Eq. \eqref{h1}.

\section{Preliminaries}\label{sec2}

Firstly,  let us recall a key inequality,  i.e., Gagliardo-Nirenberg inequality, which
 can be found in  \cite{1983W}.
Let $N \geq3$ and $2< q <2^*$, then the following sharp Gagliardo-Nirenberg inequality holds for any $u\in E$
 \begin{align}\label{h10}
|u|_q\leq C_{N,q}|u|_2^{1-\gamma_q}|\nabla u|_2^{\gamma_q}
 \end{align}
 with $ \gamma_q=\frac{N(q-2)}{2q}=\frac{N}{2}-\frac{N}{q} $, where the sharp constant $C_{N,q}$ is
$$
C_{N,q}^q=\frac{2q}{2N+(2-N)q}\left(\frac{2N+(2-N)q}{N(q-2)}  \right)^{\frac{N(q-2)}{4}}\frac{1}{|W_q|_2^{q-2}}
$$
and $W_q$ is the unique positive radial solution of equation
$$
-\Delta u+u=|u|^{q-2}u.
$$
In the special case $q =\overline{q}:=2+\frac{4}{N} $, $C_{N,\overline{q}}^{\overline{q}}=\frac{\overline{q}}{2}
  \frac{1}{|W_{\overline{q}}|_2^{\frac{4}{N}}}$,
or equivalently,
 \begin{align}\label{h46}
|W_{\overline{q}}|_2=\left(\frac{\overline{q}}{2 C_{N,\overline{q}}^{\overline{q}}}\right)^{ \frac{N}{4}}=:\overline{a}_N.
 \end{align}
For every $N\geq3$, there exists an optimal constant $\mathcal{S}>0$ depending only on $N$ such that
 \begin{align}\label{h11}
\mathcal{S}|u|_{2^*}^2\leq |\nabla u|_2^2
 \end{align}
for all $u \in  \mathcal{D}^{1,2}(\mathbb{R}^N, \mathbb{C})$. \\

Recall that  the limit problem \eqref{h1jx} was studied by Soave \cite{2021-JFA-Soave}, see also  \cite{2021CVLixinfu}, they   got  the following theorem.

 \vspace{0.2cm}
  \begin{theorem}\label{TH7}
 Let $N \geq 3$, $2+\frac{4}{N} \leq q< 2^*$, and let $a, \mu > 0$. If $q=\overline{q}:= 2+\frac{4}{N}$, we   further assume
that $\mu a^{\frac{4}{N}}< (\overline{a}_N)^{\frac{4}{N}}  $.
Then $\mathcal{I}_{q}^\infty|_{S_a}$  has a positive  real-valued ground state $u$   satisfying $ \mathcal{I}_q^\infty(u)=  m_{q,a}^ \infty $.
  \end{theorem}

  Now, we give an important property of  the ground state energy map $a\longmapsto m_{q,a}^ \infty$ with $a>0$.

 \begin{proposition}\label{TH3}
    Let $q\in[\overline{q}, 2^*)$ and $\mu>0$.   If $q=\overline{q} $, we   further assume
that $\mu a^{\frac{4}{N}}< (\overline{a}_N)^{\frac{4}{N}}  $. Then
    $a\longmapsto m_{q,a}^ \infty$ with $a>0$ is  non-increasing.
 \end{proposition}

 \begin{proof}
    Let $ 0<a<b<+\infty$,  $\theta=\frac{a}{b}<1$ and   $q\in[\overline{q}, 2^*)$. For any $u\in S_a$, setting $w=\theta^{\frac{N-2}{2}}u(\theta x)$, then, we have $\int_{\mathbb{R}^N} |w|^2dx=b^2 $, which shows that $w\in S_b$.  Similar to \cite[Lemmas 5.1 and  6.1]{2021-JFA-Soave}, there exists a unique $t_w>0$ such that $t_w\ast w\in  \mathcal{P}_{q,b}^\infty$. Moreover, by simple calculation, one obtains
 \begin{align}
 &\int_{\mathbb{R}^N} |\nabla w|^2dx=\int_{\mathbb{R}^N} |\nabla u|^2dx     \quad\quad \text{and}\quad\quad
    \int_{\mathbb{R}^N} |  w|^{2^*}dx=\int_{\mathbb{R}^N} |  u|^{2^*}dx,\label{h23}\\
    & \int_{\mathbb{R}^N} |  w|^{q}dx= \theta^{\frac{N-2}{2}q-N }\int_{\mathbb{R}^N} |  u|^{q}dx>\int_{\mathbb{R}^N} |  u|^{q}dx.\label{h24}
 \end{align}
  Then, it follows from \eqref{h23} and \eqref{h24}  that
  \begin{align*}
  m_{q,b}^\infty&\leq\mathcal{I}_q^\infty(t_w\ast w )
  = \mathcal{I}_q^\infty(e^{\frac{Nt_w}{2}}w(e^{t_w}x))\\
  &=  \frac{e^{2t_w}}{2} \int_{\mathbb{R}^N} |\nabla w|^2dx-\frac{e^{2^*t_w}}{2^*} \int_{\mathbb{R}^N} |  w|^{2^*}dx-\frac{1}{q}e^{ (\frac{q}{2}-1)Nt_w}\mu\int_{\mathbb{R}^N} |  w|^{q}dx\\
   &= \frac{e^{2t_w}}{2}\int_{\mathbb{R}^N} |\nabla u|^2dx-\frac{e^{2^*t_w}}{2^*} \int_{\mathbb{R}^N} |u|^{2^*}dx-\frac{1}{q}e^{ (\frac{q}{2}-1)Nt_w} \theta^{\frac{N-2}{2}q-N } \mu\int_{\mathbb{R}^N} |  u|^{q}dx\\
   &< \frac{e^{2t_w}}{2}\int_{\mathbb{R}^N} |\nabla u|^2dx-\frac{e^{2^*t_w}}{2^*} \int_{\mathbb{R}^N} |u|^{2^*}dx-\frac{1}{q}e^{ (\frac{q}{2}-1)Nt_w} \mu\int_{\mathbb{R}^N} |  u|^{q}dx\\
  &= \mathcal{I}_q^\infty(t_w\ast u )\\
  &\leq\max_{t>0}\mathcal{I}_q^\infty(t\ast u ),
 \end{align*}
 which and \cite[Lemma 8.1]{2021-JFA-Soave} imply that $  m_{q,b}^\infty\leq \inf_{u\in S_a}\max_{t>0}\mathcal{I}_q^\infty(t\ast u )= m_{q,a}^\infty$ by the arbitrariness of $u$.
 Thus, we complete the proof of Proposition \ref{TH3}.
 \end{proof}

 During the proofs,  the following various expressions of $\mathcal{I}_{q}(u)$ constrained on $\mathcal{P}_{q,a}$ play an important role.    In view of \eqref{h2} and   \eqref{h4},  one obtains
  \begin{align}
\mathcal{I}_{q}(u)
&= \frac{1}{2} \int_{\mathbb{R}^N} \Big[V(x)+\frac{1}{2}\langle \nabla V(x),x\rangle\Big]|u|^2dx+\Big(\frac{1}{2}-\frac{1}{2^*}\Big)\int_{\mathbb{R}^N}|u|^{2^*}dx
+\mu\Big(\frac{\gamma_q}{2}-\frac{1}{q}\Big)\int_{\mathbb{R}^N}|u|^{q}dx \label{h5}\\
&=\Big(\frac{1}{2}-\frac{1}{2^*}\Big)\int_{\mathbb{R}^N}|\nabla u|^2dx+\frac{1}{2} \int_{\mathbb{R}^N}\Big[V(x)+\frac{1}{2^*}\langle\nabla V,x\rangle\Big]|u|^2dx+\mu\Big(\frac{\gamma_q}{2^*}
-\frac{1}{q}\Big)\int_{\mathbb{R}^N}|u|^{q}dx\label{h6}\\
&=\Big(\frac{1}{2}-\frac{1}{q \gamma_q}\Big)\int_{\mathbb{R}^N}|\nabla u|^2dx+\frac{1}{2} \int_{\mathbb{R}^N}\Big[V(x)+\frac{1}{q \gamma_q}\langle\nabla V, x\rangle\Big]|u|^2dx+ \Big(\frac{1}{q \gamma_q}-\frac{1}{ 2^* }\Big)\int_{\mathbb{R}^N}|u|^{2^*}dx\label{h7}
\end{align}
with $ \gamma_q=\frac{N(q-2)}{2q}=\frac{N}{2}-\frac{N}{q} $.
For any $u\in S_a$, let
$$
(s\ast u)(x):=e^{\frac{Ns}{2}}u(e^sx),
$$
one has $s\ast u\in S_a$. Moreover,
 \begin{align}\label{h8}
\psi_u(s):&=\mathcal{I}_q(s\ast u )\nonumber\\
&=\frac{e^{2s}}{2}\int_{\mathbb{R}^N}|\nabla u|^2dx+\frac{1}{2} \int_{\mathbb{R}^N} V(e^{-s}x)|u|^2dx-\frac{e^{2^*s}}{2^*} \int_{\mathbb{R}^N}|u|^{2^*}dx
-\frac{\mu}{q}e^{(\frac{q}{2}-1)Ns}\int_{\mathbb{R}^N}|u|^{q}dx,
\end{align}
and
 \begin{align}\label{h9}
 P_q(s\ast u )=e^{2s} \int_{\mathbb{R}^N}|\nabla u|^2dx-\int_{\mathbb{R}^N} W(e^{-s}x)|u|^2dx-e^{2^*s} \int_{\mathbb{R}^N}|u|^{2^*}dx-\mu \gamma_q e^{(\frac{q}{2}-1)Ns}\int_{\mathbb{R}^N}|u|^{q}dx.
 \end{align}
Obviously, one knows that $ \psi_u'(s)= P_q(s\ast u )$, which implies that for any $u\in S_a $, $s\in \mathbb{R}$ is a critical point for $\psi_u(s)$  if and only if $s\ast u \in \mathcal{P}_{q,a}$.\\

The following lemma help us to show that that  $\mathcal{I}_q$ is bounded away from 0 on $ \mathcal{P}_{q,a}$. Here the Gagliardo-Nirenberg
inequality \eqref{h10}  plays   a crucial role.

 \begin{lemma}\label{Lem2.1}
Let $(V_2)$ be satisfied and $q\in (\overline{q}, 2^*)$. For any $u\in\mathcal{P}_{q,a} $, there exists $\delta>0$ such that $|\nabla u|_2\geq \delta$.
  \end{lemma}
 \begin{proof}
 For any $u\in\mathcal{P}_{q,a} $,
 in virtue of $(V_2)$, \eqref{h4}, Gagliardo-Nirenberg inequality \eqref{h10} and \eqref{h11}, we know that
  \begin{align}\label{h12}
 \int_{\mathbb{R}^N}|\nabla u|^2dx&=\frac{1}{2}\int_{\mathbb{R}^N} \langle \nabla V(x), x\rangle |u|^2dx+ \int_{\mathbb{R}^N}|u|^{2^*}dx+\mu \gamma_q \int_{\mathbb{R}^N}|u|^{q}dx\nonumber\\
 &\leq \sigma_2|\nabla u|_2^2+\mathcal{S}^{-\frac{N}{N-2}}|\nabla u|_2^{2^*}+\mu \gamma_q C_{N,q}^qa^{q(1-\gamma_q)}|\nabla u|_2^{q\gamma_q},
\end{align}
which implies that
$$
(1-\sigma_2)|\nabla u|_2^2\leq \mathcal{S}^{-\frac{N}{N-2}}|\nabla u|_2^{2^*}+\mu \gamma_q C_{N,q}^qa^{q(1-\gamma_q)}|\nabla u|_2^{q\gamma_q}.
$$
Since  $q\in(\overline{q}, 2^*)$,  then $ q\gamma_q>2$. Hence, for any $\mu>0$, one obtains that there exists $\delta>0$ satisfying $|\nabla u|_2\geq\delta$. So, we complete the proof.
 \end{proof}

  \begin{lemma}\label{Lem2.4}
  Assume that $(V_1)$ and $(V_2)$   hold. For any $q\in (\overline{q}, 2^*)$,
$m_{q,a}=\inf_{u\in\mathcal{P}_{q,a}} \mathcal{I}_q(u)>0$.
 \end{lemma}
  \begin{proof}
By $(V_2)$ and $P_q(u)=0$,  one infers that
 \begin{align}\label{h15}
 (1+\sigma_2)\int_{\mathbb{R}^N}|\nabla u|^2dx &\geq\int_{\mathbb{R}^N}|\nabla u|^2dx-\frac{1}{2}\int_{\mathbb{R}^N} \langle \nabla V(x), x\rangle |u|^2dx \nonumber\\
 &=\int_{\mathbb{R}^N}|u|^{2^*}dx+\mu \gamma_q \int_{\mathbb{R}^N}|u|^{q}dx\nonumber\\
 &\geq\gamma_q  \int_{\mathbb{R}^N}|u|^{2^*}dx+\mu \gamma_q \int_{\mathbb{R}^N}|u|^{q}dx.
\end{align}
Then, in view of \eqref{h2}, \eqref{h15} and $(V_1)$, we have
\begin{align} \label{h16}
\mathcal{I}_{q}(u)&=\frac{1}{2}\int_{\mathbb{R}^N}|\nabla u|^2+V(x)|u|^2dx
-\frac{1}{2^*}\int_{\mathbb{R}^N}|u|^{2^*}dx
-\frac{\mu}{q}\int_{\mathbb{R}^N}|u|^{q}dx\nonumber\\
&\geq\frac{1}{2}(1-\sigma_1) \int_{\mathbb{R}^N}|\nabla u|^2dx-\frac{1}{2^*}\int_{\mathbb{R}^N}|u|^{2^*}dx
-\frac{\mu}{q}\int_{\mathbb{R}^N}|u|^{q}dx\nonumber\\
&\geq \frac{1}{2}(1-\sigma_1) \int_{\mathbb{R}^N}|\nabla u|^2dx-\frac{1}{q}\int_{\mathbb{R}^N}|u|^{2^*}dx
-\frac{\mu}{q}\int_{\mathbb{R}^N}|u|^{q}dx\nonumber\\
&\geq \frac{1}{2}(1-\sigma_1) \int_{\mathbb{R}^N}|\nabla u|^2dx-\frac{1}{q\gamma_q}(1+\sigma_2)\int_{\mathbb{R}^N}|\nabla u|^2dx\nonumber\\
&=( \frac{1}{2}- \frac{\sigma_1}{2}-\frac{1}{q\gamma_q}-\frac{\sigma_2}{q\gamma_q})\int_{\mathbb{R}^N}|\nabla u|^2dx.
\end{align}
Then,   we deduce from Lemma \ref{Lem2.1} and \eqref{h16} that $m_{q,a}=\inf_{u\in\mathcal{P}_{q,a}} \mathcal{I}_q(u)>0$ due to the range of $\sigma_1$ and  $\sigma_2$. Thus, we complete the proof.
 \end{proof}

 \vspace{0.2cm}
Consider the decomposition of $\mathcal{P}_{q,a}$ into the disjoint union
$$
\mathcal{P}_{q,a}=\left(\mathcal{P}_{q,a} \right)^+ \cup \left(\mathcal{P}_{q,a} \right)^0\cup \left(\mathcal{P}_{q,a} \right)^-,
$$
 where
 $$
 \left(\mathcal{P}_{q,a} \right)^{+(resp.\ 0,-)}=\left\{ u\in\mathcal{P}_{q,a}: \psi_u''(0)> (resp. \ =,<)0 \right\}.
 $$

 \vspace{0.1cm}

\begin{lemma}\label{Lem2.2}
  Assume that $(V_2)-(V_3)$  hold and $q\in (\overline{q}, 2^*)$.  Then $\mathcal{P}_{q,a}= (\mathcal{P}_{q,a})^-$.
\end{lemma}
  \begin{proof}
   For any $u\in \mathcal{P}_{q,a}$, then $P_q(u)=0$, namely,
 \begin{align}\label{h13}
 \int_{\mathbb{R}^N}|\nabla u|^2dx-\frac{1}{2}\int_{\mathbb{R}^N} \langle \nabla V(x), x\rangle |u|^2dx- \int_{\mathbb{R}^N}|u|^{2^*}dx-\mu \gamma_q \int_{\mathbb{R}^N}|u|^{q}dx=0,
\end{align}
where $\gamma_q =\frac{(q-2)N}{2q} $.
From \eqref{h8}, one has
 \begin{align}\label{hjh13}
 \psi_u''(0)
 =2 \int_{\mathbb{R}^N}|\nabla u|^2dx+\int_{\mathbb{R}^N} \langle \nabla W(x), x\rangle |u|^2dx-2^*\int_{\mathbb{R}^N}|u|^{2^*}dx-\mu \gamma_q \frac{(q-2)N}{2}\int_{\mathbb{R}^N}|u|^{q}dx.
 \end{align}
Then, by \eqref{h13}, $(V_2)$ and $(V_3)$, we deduce that
 \begin{align*}
\psi_u''(0)&=\psi_u''(0)- \frac{(q-2)N}{2} P_q(u)\\
&=\Big(2-\frac{(q-2)N}{2}\Big)\int_{\mathbb{R}^N}|\nabla u|^2dx+\int_{\mathbb{R}^N} \langle \nabla W(x), x\rangle |u|^2dx\\
&\quad+\frac{(q-2)N}{4}\int_{\mathbb{R}^N} \langle \nabla V(x), x\rangle |u|^2dx+\Big (\frac{(q-2)N}{2}-2^*\Big)\int_{\mathbb{R}^N}|u|^{2^*}dx\\
&\leq\Big(2-\frac{(q-2)N}{2}\Big)\int_{\mathbb{R}^N}|\nabla u|^2dx+ \int_{\mathbb{R}^N}\Big[  \langle \nabla W(x), x\rangle + \frac{(q-2)N}{2}W(x)\Big]|u|^2dx\\
&\leq \Big(2-\frac{(q-2)N}{2}+\frac{(q-2)N}{2}\sigma_2+\sigma_3\Big)\int_{\mathbb{R}^N}|\nabla u|^2dx\\
&<0,
\end{align*}
which implies that $u\in (\mathcal{P}_{q,a})^-$. Hence, we deduce $\mathcal{P}_{q,a}= (\mathcal{P}_{q,a})^-$ and we complete the proof.
 \end{proof}

 \begin{lemma}\label{Lem2.3}
  Let $(V_2)$ and $(V_3)$  be hold and $q\in (\overline{q}, 2^*)$.
  For any $u\in S_a$,   the function  $ \psi_u$ has  a unique  critical point $s_u$. In other words,  there exists a unique $s_u\in \mathbb{R}$ such that $ s_u\ast u \in  \mathcal{P}_{q,a} $. Moreover, $\mathcal{I}_q(s_u\ast u )=\max_{s>0}\mathcal{I}_q(s\ast u )$ and $ \psi_u$   is strictly decreasing and concave on  $(s_u, +\infty)$. In particular,   $s_u<0$ if and only  if $P_{q}(u)<0$. And the map $u\in S_a\mapsto s_u\in \mathbb{R}$ is of class $C^1$.
\end{lemma}
  \begin{proof}
Since $ \psi_u'(s)= P_q(s\ast u )$, then we only  prove that $\psi_u'(s)$ has a unique root in $\mathbb{R}$. By \eqref{h8} and $(V_2)$, for any $u\in S_a$, one has
\begin{align} \label{h20}
\psi_u'(s)&=
e^{2s} \int_{\mathbb{R}^N}|\nabla u|^2dx-\int_{\mathbb{R}^N} W(e^{-s}x)|u|^2dx-e^{2^*s} \int_{\mathbb{R}^N}|u|^{2^*}dx-\mu \gamma_q e^{(\frac{q}{2}-1)Ns}\int_{\mathbb{R}^N}|u|^{q}dx\nonumber\\
&\geq (1-\sigma_2) e^{2s} |\nabla u|_2^2-e^{2^*s} \int_{\mathbb{R}^N}|u|^{2^*}dx-\mu \gamma_q e^{(\frac{q}{2}-1)Ns}\int_{\mathbb{R}^N}|u|^{q}dx.
 \end{align}
 When $ q\in (\overline{q}, 2^*) $, we deduce that $ \psi_u'(s)>0$ for $s\rightarrow-\infty$. Hence there exists $s_0\in \mathbb{R}$ such that $ \psi_u(s)$ is increasing on $(-\infty, s_0 )$. In addition, we know that $\psi_u(s)\rightarrow -\infty $ as $s\rightarrow+\infty$. Thus, there is $s_1\in \mathbb{R}$ with $s_1>s_0$ such that
 \begin{align}\label{h36}
  \psi_u(s_1)=\max_{s>0} \psi_u(s).
  \end{align}
   Therefore, $ \psi_u'(s_1)=0$, namely, $s_1\ast u\in \mathcal{P}_{q,a}$. Assume that there exists $s_2\in \mathbb{R}$ satisfying   $s_2\ast u\in \mathcal{P}_{q,a}$.   Without loss of generality, assume that $ s_2>s_1$.  From Lemma \ref{Lem2.2}, we have $ \psi_u''(s_2)<0$. Hence, there is $s_3\in(s_1, s_2)$ such that
  $$
  \psi_u(s_3)=\min_{s\in(s_1, s_2) } \psi_u(s).
  $$
  Thus, we deduce that $ \psi_u''(s_3)\geq0$ and $ \psi_u'(s_3)=0$, which implies that $s_3\ast u\in \mathcal{P}_{q,a}$ and      $s_3\ast u\in (\mathcal{P}_{q,a})^+\cup (\mathcal{P}_{q,a})^0$. This is a contradiction. Hence,  setting $s_u=s_1$, we know that $s_u\in \mathbb{R}$ is an unique number satisfying $ s_u\ast u \in \mathcal{P}_{q,a} $. And then  it follows from \eqref{h36} that $\mathcal{I}_q(s_u\ast u )=\max_{s>0}\mathcal{I}_q(s\ast u ) $. Furthermore,
  \begin{align} \label{h47}
  \psi_u'(s)<0  \quad \Leftrightarrow \quad  s>s_u.
\end{align}
   Hence,   $ \psi_u'(0)=P_{q}(u)<0$ if and only if $s_u<0$.  In the following, we show that   $ \psi_u$   is strictly decreasing and concave on  $(s_u, +\infty)$. Clearly, it follows from \eqref{h47} that  $ \psi_u$   is strictly decreasing  on  $(s_u, +\infty)$. Now, we   verify that $\psi_u''(s)<0$ on  $(s_u, +\infty)$.
  From $(V_3)$,
     \begin{align} \label{h48}
   \psi_u''(s)=&  2e^{2s} \int_{\mathbb{R}^N}|\nabla u|^2dx+\int_{\mathbb{R}^N} \langle \nabla W(e^{-s}x), e^{-s}x\rangle |u|^2dx-2^*e^{2^*s}\int_{\mathbb{R}^N}|u|^{2^*}dx-\mu q\gamma_q^2  e^{q\gamma_q s}\int_{\mathbb{R}^N}|u|^{q}dx\nonumber\\
    \leq &  (2+\sigma_3)e^{2s} \int_{\mathbb{R}^N}|\nabla u|^2dx -2^*e^{2^*s}\int_{\mathbb{R}^N}|u|^{2^*}dx-\mu q\gamma_q^2  e^{q\gamma_q s}\int_{\mathbb{R}^N}|u|^{q}dx.
   \end{align}
Let
$$
g_1(s)=(1-\sigma_2) e^{2s} |\nabla u|_2^2-e^{2^*s} \int_{\mathbb{R}^N}|u|^{2^*}dx-\mu \gamma_q e^{(\frac{q}{2}-1)Ns}\int_{\mathbb{R}^N}|u|^{q}dx
$$
and
$$
g_2(s)=(2+\sigma_3)e^{2s} \int_{\mathbb{R}^N}|\nabla u|^2dx -2^*e^{2^*s}\int_{\mathbb{R}^N}|u|^{2^*}dx-\mu q\gamma_q^2  e^{q\gamma_q s}\int_{\mathbb{R}^N}|u|^{q}dx.
$$
Obviously, $g_1$ and $g_2$ has a unique zero point. Assume that $g_1(\widetilde{s}_1)=0$ and  $g_2(\widetilde{s}_2)=0$. After a simple analysis,  it holds that $g_1(s)>0$ for  $ s\in(-\infty, \widetilde{s}_1)$ and $g_1(s)<0$ for $s\in ( \widetilde{s}_1, +\infty)$; and $g_2(s)>0 $  for $s\in (-\infty, \widetilde{s}_2)$ and $g_2(s)<0 $  for $s\in ( \widetilde{s}_2, +\infty)$. Moreover, by \eqref{h20}
  we obtain that $ \psi_u $ is strictly  increasing on $ (-\infty, \widetilde{s}_1)$.
 Since $\psi_u $ is strictly  increasing on $ (-\infty, s_u)$ and is strictly  decreasing on $ ( s_u, +\infty)$,  we infer that $\widetilde{s}_1\leq  s_u$.  If $ \widetilde{s}_2\leq s_u$, then from \eqref{h48} we have $\psi_u''(s)<0$ on  $(s_u, +\infty)$. Thus, we only need to show that $\widetilde{s}_2\leq \widetilde{s}_1 $,  which is equivalent to prove   $ g_2(\widetilde{s}_1)\leq 0$.  Then, since $g_1(\widetilde{s}_1 )=0$, namely,
$$
(1-\sigma_2) e^{2\widetilde{s}_1} |\nabla u|_2^2-e^{2^*\widetilde{s}_1} \int_{\mathbb{R}^N}|u|^{2^*}dx=\mu \gamma_q e^{(\frac{q}{2}-1)N\widetilde{s}_1}\int_{\mathbb{R}^N}|u|^{q}dx,
$$
we have
 \begin{align} \label{h50}
g_2(\widetilde{s}_1)&=(2+\sigma_3)e^{2\widetilde{s}_1} \int_{\mathbb{R}^N}|\nabla u|^2dx -2^*e^{2^*\widetilde{s}_1}\int_{\mathbb{R}^N}|u|^{2^*}dx- q\gamma_q (1-\sigma_2) e^{2\widetilde{s}_1} |\nabla u|_2^2+q\gamma_q e^{2^*\widetilde{s}_1} \int_{\mathbb{R}^N}|u|^{2^*}dx\nonumber\\
&=(2+\sigma_3-q\gamma_q (1-\sigma_2) )e^{2\widetilde{s}_1} \int_{\mathbb{R}^N}|\nabla u|^2dx+   (q\gamma_q -2^* ) e^{2^*\widetilde{s}_1}\int_{\mathbb{R}^N}|u|^{2^*}dx\nonumber\\
&<0.
 \end{align}
So, we obtain that $ \psi_u$   is    strictly     concave on  $(s_u, +\infty)$.

It remains to show that the map $u\in S_a\mapsto s_u\in \mathbb{R}$ is of class $C^1$.  Since $\pi(s,u):=\psi_u'(s)$ is of class $C^1 $, $\pi(s_u, u)=0 $ and $\partial_s\pi(s_u, u)=  \psi_u''(s_u)<0$, then the implicit function theorem gives that $u\in S_a\mapsto s_u\in \mathbb{R}$ is of class $C^1$.
 Thus, we complete the proof.
 \end{proof}

   \begin{lemma}\label{Lem2.5}
 Assume that $(V_1)$ and $( V_2)$ hold. For $q\in(\overline{q}, 2^*)$, then there is a constant $k>0$ small enough such that
 $$
 0<\sup_{\overline{A_k}} \mathcal{I}_q<m_{q,a} \quad \quad \text{and}\quad \quad u\in \overline{A_k}\Rightarrow\mathcal{I}_q(u), P_{q}(u)>0,
 $$
 where $A_k=\{u\in S_a:|\nabla u|_2^2<k\}$.
 \end{lemma}
  \begin{proof}
   It is follows from $(V_1)$, \eqref{h10}  and \eqref{h11} that
 \begin{align}\label{h18}
\mathcal{I}_{q}(u)&=\frac{1}{2}\int_{\mathbb{R}^N}|\nabla u|^2+V(x)|u|^2dx
-\frac{1}{2^*}\int_{\mathbb{R}^N}|u|^{2^*}dx
-\frac{\mu}{q}\int_{\mathbb{R}^N}|u|^{q}dx\nonumber\\
&\geq (\frac{1}{2}-\sigma_1 ) |\nabla u|_2^2 -\frac{1}{2^*}\mathcal{S}^{-\frac{N}{N-2}}|\nabla u|_2^{2^*}-\frac{\mu}{q}\gamma_qC_{N,q}^q a^{ q(1-\gamma_q)}|\nabla u|_2^{q\gamma_q}.
\end{align}
When $q\in(\overline{q}, 2^*)$, then $q\gamma_q>2$ and $\mathcal{I}_{q}(u)>0$ for any $u\in \overline{A}_k$ with $k>0$ small enough.  Besides, by $(V_2)$, one obtains that
 \begin{align}\label{h19}
P_q(u)&=\int_{\mathbb{R}^N}|\nabla u|^2dx-\frac{1}{2}\int_{\mathbb{R}^N} \langle \nabla V(x), x\rangle |u|^2dx- \int_{\mathbb{R}^N}|u|^{2^*}dx-\mu \gamma_q \int_{\mathbb{R}^N}|u|^{q}dx\nonumber\\
&\geq (1-\sigma_2)\int_{\mathbb{R}^N}|\nabla u|^2dx-\mathcal{S}^{-\frac{N}{N-2}}| \nabla u|_2^{2^*}-\mu \gamma_q C_{N,q}^q a^{q(1- \gamma_q )} | \nabla u|_2^{q \gamma_q},
\end{align}
which shows that $P_q(u)>0$ for any $u\in \overline{A}_k$ with $k>0$ small enough.   Moreover, choosing $k$ sufficient small, from Lemma \ref{Lem2.4} we have, for all $u\in \overline{A}_k$,
$$
\mathcal{I}_{q}(u)\leq \frac{1}{2}\int_{\mathbb{R}^N}|\nabla u|^2dx < m_{q,a}.
$$
Hence, we complete the proof of this lemma.
 \end{proof}

   \begin{lemma}\label{Lem2.9}
Assume that $(V_1)$ holds  and $q\in [ \overline{q}, 2^*)$.    If $q=\overline{q}$, we   further assume
that $\mu a^{\frac{4}{N}}< (\overline{a}_N)^{\frac{4}{N}}  $.   Then $m_{q,a}<m_{q,a}^ \infty$.
 \end{lemma}
  \begin{proof}
According to  Theorem \ref{TH7}, there exists    $0<u\in S_a$  such that   $I_q^\infty(u)=  m_{q,a}^ \infty$. Then $u\in \mathcal{P}_{q,a}^{\infty}$.
Moreover,  from Lemma \ref{Lem2.3},    there exists  a unique $s_u\in \mathbb{R}$ such that $ s_u\ast u\in \mathcal{P}_{q,a} $. Hence, since $V(x)\not\equiv 0$
 and $\sup_{x\in\mathbb{R}^N} V(x) =0$, it is easy to know that
  $$
  m_{q,a}\leq I_q(s_u\ast u )=I_q^\infty(s_u\ast u ) + \frac{1}{2} \int_{\mathbb{R}^N} V(e^{-s}x)|u|^2dx< I_q^\infty(s_u\ast u )\leq I_q^\infty(  u ),
  $$
 where  \cite[Lemmas 5.1 amd 6.1]{2021-JFA-Soave}  is applied in the last inequality.       Thus, $m_{q,a}<m_{q,a}^ \infty$.
 So, we complete the proof.
   \end{proof}

\begin{lemma}\label{Lem4.1}
 If $ u \in \mathcal{P}_{q,a}$ is a critical point for $\mathcal{I}_q|_{\mathcal{P}_{q,a} }$ with $q\in[\overline{q}, 2^*)$ and $\left(\mathcal{P}_{q,a} \right)^{ 0}=\emptyset$, then $u$ is a critical point for $\mathcal{I}_q|_{S_a } $.
\end{lemma}

\begin{proof}
If $ u \in \mathcal{P}_{q,a}$ is a critical point for $\mathcal{I}_q|_{\mathcal{P}_{q,a} }$, then there exist $ \lambda, \nu\in \mathbb{R}$ such that
$$
\mathcal{I}_q'(u)+\lambda u+\nu P_q'(u)=0,
$$
namely,
$$
(1+2\nu)(-\Delta u)+\lambda u-(1+2^* \nu )|u|^{2^*-2}u-\mu(1+ \nu \gamma_q q)|u|^{q-2}u+(V(x)-2 \nu W(x))u=0 \ \ \ \text{in} \ \mathbb{R}^N.
$$
The corresponding Poho\u{z}aev identity is given by
\begin{align}\label{h58}
(1+2\nu) |\nabla u|_2^2-\frac{1}{2}\int_{\mathbb{R}^N}\langle\nabla V, x\rangle |u|^2dx+\nu\int_{\mathbb{R}^N}\langle\nabla W, x\rangle |u|^2dx-(1+2^* \nu ) |u|_{2^*}^{2^*}-\mu \gamma_q(1+ \nu \gamma_q q)|u|_{q}^{q}=0.
\end{align}
Moreover, since $u\in \mathcal{P}_{q,a}$, then $P_q(u)=0$, that is,
\begin{align}\label{h57}
|\nabla u|_2^2=\frac{1}{2}\int_{\mathbb{R}^N} \langle \nabla V(x), x\rangle |u|^2dx+ \int_{\mathbb{R}^N}|u|^{2^*}dx+\mu \gamma_q \int_{\mathbb{R}^N}|u|^{q}dx.
\end{align}
Combining \eqref{h58} with \eqref{h57}, one infers that
$$
\nu \big(2|\nabla u|_2^2- 2^*|u|_{2^*}^{2^*}-\mu q \gamma_q^2 |u|_{q}^{q}+\int_{\mathbb{R}^N}\langle\nabla W, x\rangle |u|^2dx\big )=0.
$$
  From Lemma \ref{Lem2.2}, we get $\left(\mathcal{P}_{q,a} \right)^{ 0}=\emptyset$, namely,
$$
2|\nabla u|_2^2- 2^*|u|_{2^*}^{2^*}-\mu q \gamma_q^2 |u|_{q}^{q}+\int_{\mathbb{R}^N}\langle\nabla W, x\rangle |u|^2dx\neq0,
$$
 which implies that $\nu=0$. Hence, we complete the proof.
\end{proof}

The following Lemma  can be found in   Br\'{e}zis and Kato   \cite[Lemma 2.1]{1979-JMPA} (see also \cite[Lemma 2.3]{2014-zhang}), which is   crucial for the $L^\infty$  estimate for the solution of Eq. \eqref{h1}.

\begin{lemma} \label{Lem4.2}
Let $\Omega \subset  \mathbb{R}^N$ and $ k(x)\in  L^{\frac{N }{2}}(\Omega )$ be a nonnegative function. Then, for every
$\varepsilon> 0$, there exists a constant $C(\varepsilon, k) > 0$ such that
$$
\int_{\Omega}k(x) u^2dx\leq \varepsilon \int_{\Omega}|\nabla u|^2dx+ C(\varepsilon, k) \int_{\Omega}| u|^2dx
$$
for all $u\in H^1(\Omega)$.
\end{lemma}

\begin{lemma}\label{Lem4.3}
Assume that $(V_1)$ and $(V_4)$ or $(\widetilde{V}_1)$ and $(V_4)$ hold.
  Let $   u\in S_a$ be a nonnegative real-valued solution for Eq. \eqref{h1} with $\mu, a>0$  and  $q\in [\overline{q}, 2^*)$, then we get the associated Lagrange multiplier $\lambda>0$, and there exists $C>0$ such that  $|u|_\infty <C $.
\end{lemma}

\begin{proof} First, we prove   that if $0\leq u\in S_a$ is a real-valued solution for Eq. \eqref{h1}, then the associated Lagrange multiplier $\lambda>0$. This follows
easily by testing Eq.  \eqref{h1}  with $u$, and using the fact that $P_q(u)=0$ and $(V_4)$:
\begin{align}\label{h85}
\lambda a^2=(1-\gamma_q)\mu|u|_q^q-\int_{\mathbb{R}^N}(V(x)+W(x))|u|^2dx>0
 \end{align}
because $\gamma_q<1$ and $a>0$.  Next, we show that $|u|_\infty <C $ for some $C>0$.
Let   $A_{m}=\{x\in\mathbb{R}^{N}:|u(x)|\leq m\}$, $B_{m}=\mathbb{R}^{N}\backslash{A_{m}}$, where  $m\in \mathbb{N}$.
For any $p>0$, define
\begin{equation*} u_m=
\begin{cases}
  |u|^{2p+1}, \quad\quad &x\in A_{m}\\
 m^{2p}u, \quad\quad&x\in B_{m}
\end{cases}
\end{equation*}
and
\begin{equation*} w_m=
\begin{cases}
|u|^{p+1},\quad\quad &x\in A_{m}\\
 m^p u,\quad\quad &x\in B_{m}.\\
\end{cases}
\end{equation*}
By simple calculation, one has
\begin{equation*} \nabla u_m=
\begin{cases}
 (2p+1)|u|^{2p}\nabla u, \quad\quad &x\in A_{m}\\
 m^{2p}\nabla u, \quad\quad&x\in B_{m}\\
\end{cases}
\end{equation*}
and
\begin{equation*} \nabla w_m=
\begin{cases}
 (p+1)|u|^{ p}\nabla u, \ \ \ \ \ \ \ &x\in A_{m}\\
m^p\nabla u, \ \ \ \ \ \ \ \ \ \ \ \ \ \ &x\in B_{m}.\\
\end{cases}
\end{equation*}
Notice that $u_{m}$, $w_{m}\in H^{1}(\mathbb{R}^{N})$ and $w^{2}_{m}=uu_{m}\leq|u|^{2(p+1)}$.
Taking $u_{m}$ as a test function in  $\langle \mathcal{I}_q(u), u_m\rangle=0$, we obtain
$$
\int_{\mathbb{R}^N} \nabla u \cdot \nabla u_m+ \lambda u u_m+V(x) u u_m dx=\int_{\mathbb{R}^N} |u|^{2^*-2} u u_m+\mu|u|^{q-2} u u_mdx,
$$
which and $\lambda>0$ imply that
\begin{align}\label{h59}
\int_{A_m} (2p+1) |u|^{2p} |\nabla u|^2dx+\int_{B_m}  m^{2p} |\nabla u|^2dx \leq \int_{\mathbb{R}^N} |V(x)| u u_m  + |u|^{2^*-2} u u_m+\mu|u|^{q-2} u u_mdx.
\end{align}
Since
\begin{align}\label{h60}
  \int_{\mathbb{R}^N} |\nabla w_m|^2dx=(p+1)^2 \int_{A_m}   |u|^{2p} |\nabla u|^2dx+ \int_{B_m}  m^{2p} |\nabla u|^2dx,
\end{align}
then combining \eqref{h59} with \eqref{h60}, by $(V_1)$ or  $(\widetilde{V}_1)$  and Lemma \ref{Lem4.2}, one has
\begin{align*}
\frac{2p+1}{(p+1)^2}  \int_{\mathbb{R}^N} |\nabla w_m|^2dx&=(2p+1)\int_{A_m}   |u|^{2p} |\nabla u|^2dx+\frac{2p+1}{(p+1)^2}  \int_{B_m}  m^{2p} |\nabla u|^2dx\\
&\leq \int_{\mathbb{R}^N} |V(x)| u u_m  + |u|^{2^*-2} u u_m+\mu|u|^{q-2} u u_mdx\\
&\leq \int_{\mathbb{R}^N}( C_1+|u|^{2^*-2}+\mu|u|^{q-2} ) w_m^2dx\\
&\leq \varepsilon \int_{\mathbb{R}^N}   |\nabla w_m|^2dx+ C(\varepsilon,u,\mu) \int_{\mathbb{R}^N}w_m^2dx.
\end{align*}
Hence, there exists $C_2$, depending on $\varepsilon, q, u, \mu$, such that
\begin{align}\label{h87}
\int_{\mathbb{R}^N} |\nabla w_m|^2dx \leq C_2 \int_{\mathbb{R}^N}w_m^2dx.
\end{align}
Furthermore, since $ w_m=|u|^{p+1}$ in $A_m$ and $w_m^2\leq |u|^{2(p+1)}$ on $\mathbb{R}^N$, it follows from \eqref{h87} that
\begin{align*}
\Big(\int_{A_m} |  u|^{2^*(p+1)}dx\Big)^{ \frac{2}{2^*}}
= \Big(\int_{A_m} |  w_m|^{2^*}dx\Big)^{ \frac{2}{2^*}}
\leq& \Big(\int_{\mathbb{R}^N} |  w_m|^{2^*}dx\Big)^{ \frac{2}{2^*}}\\
\leq & \int_{\mathbb{R}^N} |\nabla w_m|^2dx \\ \leq& C_2 \int_{\mathbb{R}^N}|u|^{2(p+1)}dx.
\end{align*}
Let $m\rightarrow\infty$, we infer that
$$
\Big(\int_{\mathbb{R}^N} |  u|^{2^*(p+1)}dx\Big)^{ \frac{2}{2^*}}\leq C_2 \int_{\mathbb{R}^N}|u|^{2(p+1)}dx,
$$
that is, for any $p\geq 2$, it holds that $u\in L^{2(p+1)} (\mathbb{R}^N )$ and $u\in L^{2^*(p+1)} (\mathbb{R}^N )$, and     there exists $C_3>0$ satisfying
\begin{align}\label{h61}
|u|_{ {2^*(p+1) }}\leq C_3^{ \frac{1}{2(p+1)}}|u|_{ {2(p+1) } }.
\end{align}
Setting  $ \beta:=\frac{2^*}{2}$, we have $\beta>1$. Taking $p+1=\beta^j $ for $j\in N$, then \eqref{h61} changes
$$
|u|_{2^* \beta^j } \leq C_3^{ \frac{1}{2 \beta^j}}|u|_{2^*\beta^{j-1} }.
$$
Further, we proceed the $j$ times iterations that
$$
|u|_{2^* \beta^j } \leq C_3^{\sum_{i=1}^j \frac{1}{2 \beta^i}}|u|_{2^* }.
$$
Let $j\rightarrow+\infty$, by Sobolev inequality and the fact that $|\nabla u|_2< C_4$ and $\lim_{j\rightarrow\infty}C_3^{\sum_{i=1}^j \frac{1}{2 \beta^i}}<\infty $, there exists $C>0$, depending on $\varepsilon, p, u,\mu$,   such that
$$
|u|_{\infty}\leq C.
$$
Thus, we complete the proof.
\end{proof}

\section{ Proof of   Theorem \ref{TH1}}\label{sec3}
 Now,
we are going to  prove the existence  of the positive   ground state normalized solution
  for Eq. \eqref{h1} with $q\in (\overline{q}, 2^*)$.\\

 {\bf {Proof of Theorem \ref{TH1}}}:   For any $q\in (\overline{q}, 2^*)$, we  first     show that  there exists  $(\lambda, u) \in \mathbb{R}^+ \times E$  that
 solves Eq. \eqref{h1} satisfying
 $\mathcal{I}_q(u)=m_{q,a}$.
 We  are going to do it in three steps.

 {{\bf Step 1:}} we are looking for a couple $(\lambda, u) \in \mathbb{R}^+ \times E$ to satisfy
 $$
  - \Delta u+V(x)u+\lambda u=|u|^{2^*-2}u+\mu |u|^{q-2}u \quad \quad \text{in} \ \mathbb{ R}^N.
$$
 The arguments are similar to \cite{2021-JFA-Soave}. For the reader's convenience, we present the proof for  it.  Let $k>0$ be defined by Lemma \ref{Lem2.5}.
 Considering the augmented functional $\widetilde{\mathcal{I}}_q:\mathbb{R}\times E\rightarrow \mathbb{R}$ defined by
  \begin{align}\label{h25}
\widetilde{\mathcal{I}}_q(s,u):=\mathcal{I}_q(s\ast u)=&\frac{e^{2s}}{2}\int_{\mathbb{R}^N}|\nabla u|^2dx+\frac{1}{2} \int_{\mathbb{R}^N} V(e^{-s}x)|u|^2dx\nonumber\\
&-\frac{e^{2^*s}}{2^*} \int_{\mathbb{R}^N}|u|^{2^*}dx
-\frac{\mu}{q}e^{(\frac{q}{2}-1)Ns}\int_{\mathbb{R}^N}|u|^{q}dx.
  \end{align}
 Notice that $\widetilde{\mathcal{I}}_q$ is of class $C^1$. Denoting by $\mathcal{I}_q^c$ the closed sub-level set $\{u \in S_a: \mathcal{I}_q(u)\leq c\}$, we introduce the minimax class
  \begin{align}\label{h26}
\Gamma:=\left\{\gamma=(\alpha,\beta)\in C([0,1], \mathbb{R}\times S_a): \gamma(0)\in (0, \overline{A}_k), \gamma(1)\in (0,  \mathcal{I}_q^0) \right\},
  \end{align}
with associated minimax level
  \begin{align}\label{h27}
 \eta_{q,a}:=\inf_{\gamma\in \Gamma}\max_{(s,u)\in \gamma([0,1])} \widetilde{\mathcal{I}}_q(s,u).
  \end{align}
  Let $u\in S_a$. Since $|\nabla (s\ast u)|_2\rightarrow 0^+$ as $s\rightarrow -\infty$ and $  \mathcal{I}_q(s\ast u)\rightarrow -\infty$ as  $s\rightarrow +\infty$, then there exists $s_0<<-1$ and $s_1>>1$ such that
  \begin{align}\label{h28}
  \gamma_u: \tau\in[0,1]\mapsto (0, ((1-\tau)s_0+\tau s_1)\ast u)\in \mathbb{R}\times S_a
   \end{align}
  is a path in $\Gamma$. Then $\eta_{q,a}$ is a real number. Now, for any $\gamma=(\alpha,\beta)\in \Gamma$, let us consider the function
  $$
  P_q\circ\gamma:\tau\in[0,1] \mapsto P_q(\alpha(\tau)\ast \beta(\tau))\in \mathbb{R},
  $$
we have  $P_q\circ\gamma(0) =P_q(\beta(0)) >0$ since $\beta(0)\in\overline{ A}_k$ by Lemma \ref{Lem2.5}.   Now we claim that  $P_q\circ\gamma(1) =P_q(\beta(1)) <0$. In fact, it follows from  \eqref{h8} and \eqref{h20} that
 $ \psi_u(s)>0$ for $s\rightarrow-\infty$ and $ \psi_u'(s)>0$ for $s\rightarrow-\infty$.
Hence, from Lemma \ref{Lem2.3} we obtain $ \psi_{\beta(1)}(s)>0$ for $s\leq s_{\beta(1)}$. Moreover, $ \psi_{\beta(1)}(0)=\mathcal{I}_q(\beta(1))\leq 0$ because $\beta(1)\in \mathcal{I}_q^0$, which implies that $s_{\beta(1)}<0$. Thus from Lemma \ref{Lem2.3} we get the claim. Furthermore, the map $\tau\mapsto  \alpha(\tau)\ast \beta(\tau)$ is continuous from $[0, 1]$ to $E$, and hence we deduce that there exists  $\tau_\gamma\in(0, 1)$ such that $P_q\circ\gamma(\tau_\gamma ) =0$, namely $\alpha(\tau_\gamma)\ast \beta(\tau_\gamma)\in \mathcal{P}_{q,a}$; this implies that
$$
\max_{\gamma([0,1])} \widetilde{\mathcal{I}}_q\geq  \widetilde{\mathcal{I}}_q( \gamma(\tau_\gamma))=\widetilde{\mathcal{I}}_q(\alpha(\tau_\gamma), \beta(\tau_\gamma) )= \mathcal{I}_q(\alpha(\tau_\gamma)\ast \beta(\tau_\gamma) )\geq \inf_{u\in\mathcal{P}_{q,a} } \mathcal{I}_q=m_{q,a}.
$$
Hence, $ \eta_{q,a}\geq m_{q,a}$. Besides, if $u\in \mathcal{P}_{q,a}  $, then $\gamma_u$  defined in \eqref{h28} is a path in $\Gamma$ with
$$
\mathcal{I}_q(u)=\max_{\gamma_u([0,1]) }\widetilde{\mathcal{I}}_q\geq \eta_{q,a},
$$
whence the reverse inequality $m_{q,a} \geq \eta_{q,a}$ follows.  Combining this with Lemma \ref{Lem2.5}, we infer that
$$
\eta_{q,a}=m_{q,a}>\sup_{( \overline{A}_k\cup \mathcal{I}_q^0)\cap S_a}\mathcal{I}_q=\sup_{( (0,\overline{A}_k)\cup (0,\mathcal{I}_q^0))\cap (\mathbb{R}\times S_a)} \widetilde{\mathcal{I}}_q.
$$
Using the terminology in \cite[Section 5]{1993book}, this means that $\{\gamma([0, 1]) :\gamma\in \Gamma\}$ is a homotopy stable family of compact subsets of $\mathbb{R} \times S_a$ with extended closed boundary $(0, \overline{A}_k)\cup (0,  \mathcal{I}_q^0)$, and that the superlevel set $\{ \mathcal{I}_q\geq \eta_{q,a}\}$ is a dual set for $\Gamma$.  Therefore, applying \cite[Theorem 5.2]{1993book},  taking any minimizing sequence ${\gamma_n=(\alpha_n, \beta_n)}\subset \Gamma_n$ for $\widetilde{\mathcal{I}}_q|_{\mathbb{R}\times S_a}$  at the level   $\eta_{q,a} $ with the property that $\alpha_n=0$ and $\beta_n(\tau) \geq0$ a.e. in $\mathbb{R}^N$ for every $\tau\in [0, 1]$, there exists a   sequence $\{(s_n, w_n)\}\subset \mathbb{R} \times S_a\backslash ((0, \overline{A}_k)\cup (0,  \mathcal{I}_q^0))$  such that, as  $n\rightarrow\infty$,
\begin{itemize}
  \item   [$(i)$]  $\widetilde{\mathcal{I}}_q(s_n, w_n) \rightarrow\eta_{q,a}$;
  \item   [$(ii)$] $\widetilde{\mathcal{I}}'_q|_{\mathbb{R}\times S_a}(s_n, w_n)\rightarrow 0$;
  \item   [$(iii)$] $dist((s_n, w_n), (0,\beta_n(\tau)))\rightarrow 0$.
\end{itemize}
Let $u_n:=s_n\ast w_n= e^{\frac{Ns_n}{2}}w_n(e^{s_n}x) $. It follows from $(i)$ that
 \begin{align}\label{h45}
\lim_{n\rightarrow\infty} \mathcal{I}_q(u_n)=\lim_{n\rightarrow\infty} \mathcal{I}_q( s_n\ast w_n )= \lim_{n\rightarrow\infty} \widetilde{\mathcal{I}}_q( s_n, w_n )= \eta_{q,a}=m_{q,a}.
 \end{align}
Moreover, from $(ii)$ and \eqref{h25} we know that
 \begin{align*}
\partial_s \widetilde{\mathcal{I}}_q( s_n, w_n )=\langle\widetilde{\mathcal{I}}_q( s_n, w_n ), (1,0) \rangle\rightarrow 0 \quad \quad \text{as}\ n\rightarrow\infty,
 \end{align*}
where
\begin{align*}
\partial_s \widetilde{\mathcal{I}}_q( s_n, w_n )=&e^{2s_n} \int_{\mathbb{R}^N}|\nabla w_n|^2dx-\int_{\mathbb{R}^N} W(e^{-s_n}x)|w_n|^2dx\\
&-e^{2^*{s_n}} \int_{\mathbb{R}^N}|w_n|^{2^*}dx-\mu \gamma_q e^{(\frac{q}{2}-1)Ns_n}\int_{\mathbb{R}^N}|w_n|^{q}dx.
 \end{align*}
Hence, by \eqref{h9}, one gets
\begin{align}\label{h41}
P_q(u_n)=P_q(s_n\ast w_n )=\partial_s \widetilde{\mathcal{I}}_q( s_n, w_n ) \rightarrow 0 \quad \quad \text{as}\ n\rightarrow\infty.
 \end{align}
Besides, setting $h_n\in T_{u_n}:=\{z\in E: \langle z, u_n\rangle_{L^2}=\int_{\mathbb{R}^N} z  u_ndx= 0\} $, by simple calculate, we have
 \begin{align*}
&\left \langle\mathcal{I}'_q(u_n ),h_n\right\rangle\\
= &
\int_{\mathbb{R}^N} \nabla u_n\cdot \nabla h_n +V(x) u_n h_ndx
- \int_{\mathbb{R}^N}|u_n|^{2^*-2}u_n h_ndx-\mu\int_{\mathbb{R}^N}|u_n|^{q-2}u_n h_ndx\\
 =& e^{2s_n} \int_{\mathbb{R}^N}  e^{-\frac{N+2}{2}s_n}\nabla w_n\cdot \nabla h_n(e^{-s_n}x)dx+ \int_{\mathbb{R}^N}  V(e^{-s_n}x)w_n e^{-\frac{N}{2}s_n}h_n(e^{-s_n}x)dx\\
 &-e^{2^*s_n}  \int_{\mathbb{R}^N}|w_n|^{2^*-2}w_n e^{-\frac{Ns_n}{2}}h_n(e^{-s_n}x)dx-\mu e^{(\frac{q}{2}-1)Ns_n}\int_{\mathbb{R}^N}|w_n|^{q-2}w_n e^{-\frac{Ns_n}{2}}h_n(e^{-s_n}x)dx\\
 =&  e^{2s_n} \int_{\mathbb{R}^N}   \nabla w_n\cdot \nabla \widetilde{h}_ndx+ \int_{\mathbb{R}^N}  V(e^{-s_n}x)w_n \widetilde{h}_ndx \\
  &-e^{2^*s_n}  \int_{\mathbb{R}^N}|w_n|^{2^*-2}w_n \widetilde{h}_ndx-\mu e^{(\frac{q}{2}-1)Ns_n}\int_{\mathbb{R}^N}|w_n|^{q-2}w_n \widetilde{h}_ndx,
 \end{align*}
 which implies that
\begin{align}\label{h42}
\left \langle\mathcal{I}'_q(u_n ),h_n\right\rangle=\big \langle\widetilde{\mathcal{I}}'_q(s_n, w_n ), (0,\widetilde{h}_n)\big\rangle,
 \end{align}
where $\widetilde{h}_n(x)=e^{-\frac{Ns_n}{2}}h_n(e^{-s_n}x).$ Now, we claim that
\begin{align}\label{h43}
(0,\widetilde{h}_n)\in \widetilde{T}_{(s_n, w_n)}:=\left\{(z_1, z_2)\in\mathbb{ R}\times E:\langle w_n, z_2\rangle_{L^2}=0\right \}.
 \end{align}
In fact,
\begin{align*}
(0,\widetilde{h}_n)\in \widetilde{T}_{(s_n, w_n)}&\Leftrightarrow \langle w_n, \widetilde{h}_n\rangle_{L^2}=0\\
&\Leftrightarrow  \int_{\mathbb{R}^N} w_n e^{-\frac{Ns_n}{2}}h_n(e^{-s_n}x)dx =0\\
&\Leftrightarrow  \int_{\mathbb{R}^N}  e^{\frac{Ns_n}{2}}w_n(e^{s_n}x)h_n(x)dx =0\\
&\Leftrightarrow \langle u_n, h_n\rangle_{L^2}=0.\\
&\Leftrightarrow  h_n\in T_{u_n}.
 \end{align*}
Therefore, in view of $(ii)$, \eqref{h42}  and  \eqref{h43}, one has, as $n\rightarrow\infty$,
\begin{align*}
\left \langle\mathcal{I}'_q(u_n ),h_n\right\rangle\rightarrow 0 \quad \quad \text{for all}\ h_n\in T_{u_n}.
 \end{align*}
 That is,
 \begin{align} \label{h44}
 \mathcal{I}_q'|_{S_a}(u_n)\rightarrow 0  \quad \quad \text{as}\ n\rightarrow\infty.
 \end{align}
 Furthermore,
  using the last  item  $(iii)$,  it holds that $\{s_n\}$ is bounded from above and from below.  Thus,  in virtue of \eqref{h45}, \eqref{h41} and \eqref{h44},    we find that  $\{u_n\}\subset S_a$ is a Palais-Smale sequence for $\mathcal{I}_q|_{S_a}$  at the level $\eta_{q,a}=m_{q,a}$     satisfying
 \begin{align}\label{h32}
\mathcal{I}_q
(u_n)\rightarrow m_{q,a},\quad \quad \mathcal{I}_q'|_{S_a}(u_n)\rightarrow 0 \quad \quad \text{and}\quad \quad P_q(u_n) \rightarrow0   \quad \quad \text{as}\ n\rightarrow\infty.
 \end{align}
 Similar to \eqref{h16},   we show that $\{u_n\}$ is bounded in $E $. Then, up to a subsequence, there is $u\in E$ such that, up to a subsequence,
 \begin{align*}
 &u_n\rightharpoonup u \quad \quad\quad\quad\quad \text{in}\ E;\\
 & u_n \rightarrow u \quad \quad \quad\quad\quad\text{in}\ L_{loc}^p(\mathbb{R}^N), p\in(2,2^*);\\
  & u_n(x) \rightarrow u(x)\quad \quad  \ a.e. \ \text{on}\ \mathbb{R}^N.
 \end{align*}
 Besides, by
 \eqref{h32}  and the Lagrange multipliers rule, there exists $\lambda_n \in \mathbb{R}$ such that
\begin{align*}
  Re  \int_{\mathbb{R}^N} \nabla u_n\cdot\nabla \overline{\varphi}   +V(x)u_n \overline{\varphi}   + \lambda_n  u_n \overline{\varphi}
-  |u_n|^{2^*-2}u _n\overline{\varphi}
-\mu |u_n|^{q-2}u_n \overline{\varphi}  dx= o(1) \|\varphi\|
\end{align*}
 for every $ \varphi\in E$, where  $\overline{\varphi}$ is the complex conjugate of $\varphi$ and $Re$ denotes the real part. Choosing $ \varphi=u_n$ in the above equality, it provides that
 \begin{align}\label{h33}
\lambda_n a^2=-|\nabla u_n|_2^2-\int_{\mathbb{R}^N} V(x)|u_n|^2dx+|u_n|_{2^*}^{2^*}+\mu |u_n|_{q}^{q}+o(1).
 \end{align}
From the boundedness of $\{u_n\}$  in $E$, we know that $ \{ \lambda_n\}$ is  bounded in $ \mathbb{R}  $. Then there is $\lambda\in \mathbb{R}$ satisfying $\lambda_n\rightarrow \lambda$ as $n\rightarrow\infty$.   Therefore, $(\lambda, u) \in \mathbb{R}  \times E$  satisfies
$$
  - \Delta u+V(x)u+\lambda u=|u|^{2^*-2}u+\mu |u|^{q-2}u \quad \quad \text{in} \ \mathbb{ R}^N.
$$
In the following, we prove that $\lambda>0$.
Since $P_q(u_n)\rightarrow 0$ as $n\rightarrow\infty$, namely,
 \begin{align*}
\int_{\mathbb{R}^N}|\nabla u_n|^2dx=\int_{\mathbb{R}^N}   W(x) |u_n|^2dx+ \int_{\mathbb{R}^N}|u_n|^{2^*}dx+\mu \gamma_q \int_{\mathbb{R}^N}|u_n|^{q}dx+o(1),
 \end{align*}
which and  \eqref{h33} imply that
\begin{align}\label{h34}
\lambda_n a^2= (1-\gamma_q)\mu\int_{\mathbb{R}^N}|u_n|^qdx-\int_{\mathbb{R}^N} \left(V(x)+W(x)\right)|u_n|^2dx+o(1),
 \end{align}
 where $\gamma_q<1 $.
Then, by $(V_4)$ and Fatou lemma we infer that
$$
\lambda a^2\geq (1-\gamma_q) \mu\lim_{n\rightarrow\infty}\int_{\mathbb{R}^N}|u_n|^qdx\geq (1-\gamma_q)\mu\int_{\mathbb{R}^N}|u|^qdx,
$$
which implies that $\lambda\geq 0$. Moreover, if $u\neq 0$, then we deduce that $\lambda> 0$. Next, we claim that $u\neq 0$. Assume that $u=0$. Then it follows from $(V_1)$ and $(V_2)$  that
$$
m_{q,a}=\mathcal{I}_q(u_n)=\mathcal{I}_q^\infty(u_n)+o(1) \quad \quad \text{and}\quad \quad
P_q(u_n) =P_q^\infty(u_n)+o(1),
$$
 which shows that $P_q^\infty(u_n)\rightarrow 0$ as $n\rightarrow\infty$. Hence, for any $n\in \mathbb{N}$, there exists a unique $t_n>0$ with $t_n\rightarrow 0$ as $n\rightarrow\infty$ such that $ t_n\ast u_n\in \mathcal{P}_{q,a}^\infty$. Then
 $$
 m_{q,a}^\infty\leq \mathcal{I}_q^\infty(t_n\ast u_n)=m_{q,a}+o(1),
 $$
which is contradiction with  Lemma \ref{Lem2.9}. Thus, $u\neq0$. This shows that $\lambda>0$. So  $(\lambda, u) \in \mathbb{R}^+ \times E$  satisfies
$$
  - \Delta u+V(x)u+\lambda u=|u|^{2^*-2}u+\mu |u|^{q-2}u \quad \quad \text{in} \ \mathbb{ R}^N.
$$
Consequently, $P_q(u)=0$.

  { \bf{Step 2:}} we  prove that $u\in S_a$. Let $|u|_2=b\in(0,a]$, then  $u\in \mathcal{P}_{q,b} $. Similar to Lemma \ref{Lem2.1} and \eqref{h16}, we get $\mathcal{I}_q(u)> 0$. If $a\neq b$, setting $c^2=a^2-b^2\in(0,a^2)$ and $ \upsilon_n:=u_n-u\rightharpoonup 0$ in $E$ as $n\rightarrow\infty$, then $ \| \upsilon_n\|_2=c>0$.   By  Br\'{e}zis-Lieb Lemma \cite{1983Wi} and the fact that $P_q(u)=0$, one infers
\begin{align}
 &\mathcal{I}_q(u_n)=\mathcal{I}_q( \upsilon_n)+\mathcal{I}_q(u)+o(1)=\mathcal{I}_q^\infty( \upsilon_n)+\mathcal{I}_q(u)+o(1),\label{h35}\\
 &P_q(u_n)=P_q(\upsilon_n)+P_q(u)+o(1)= P_q^\infty(\upsilon_n)+ o(1)=P_q^\infty(\upsilon_n)+o(1),\nonumber
 \end{align}
which implies $P_q^\infty(\upsilon_n) \rightarrow 0 $ as $n\rightarrow\infty$. Hence, for any $n\in \mathbb{N}$, there exists a unique $\widetilde{t}_n>0$ with $\widetilde{t}_n\rightarrow 0$ as $n\rightarrow\infty$ such that $ \widetilde{t}_n\ast \upsilon_n\in \mathcal{P}_{q,c}^\infty$. So, it follows from \eqref{h35} and $\mathcal{I}_q(u)>0$ that
$$
m_{q,c}^\infty\leq \mathcal{I}_q^\infty(\widetilde{t}_n\ast \upsilon_n)=\mathcal{I}_q^\infty(  \upsilon_n)+o(1)=m_{q,a}-\mathcal{I}_q(u)+o(1),
$$
which and Lemma \ref{Lem2.9} imply that $m_{q,c}^\infty\leq  m_{q,a}<m_{q,a}^\infty$.   That is contradiction with Proposition \ref{TH3}. Hence, $a=b$, namely, $u\in S_a$.

 { \bf{Step 3:}} we show that $ \mathcal{I}_q(u)=m_{q,a} $.
 Since $u_n\rightharpoonup u$ in $L^2(\mathbb{R}^N, \mathbb{C})$ as $n\rightarrow\infty$ and $u\in S_a$,  one has, as $n\rightarrow\infty$,
\begin{align*}
 u_n\rightarrow u \quad \quad \text{in} \  L^2(\mathbb{R}^N, \mathbb{C}).
\end{align*}
    Then, by Gagliardo-Nirenberg inequality \eqref{h10}, we have,  as $n\rightarrow\infty$,
    \begin{align*}
 u_n\rightarrow u \quad \quad \text{in} \  L^t(\mathbb{R}^N, \mathbb{C})
\end{align*}
  with $t\in(2, 2^*)$.
      In virtue of  $u\in \mathcal{P}_{q,a} $, one infers $\mathcal{I}_q(u)\geq  m_{q,a}$.
 Hence, combining \eqref{h6} with $P_q(u)=0$, by Fatou lemma and the fact that $u_n\rightarrow u$ in $L^p(\mathbb{R}^N, \mathbb{C})$ with $p\in[2, 2^*)$ as $n\rightarrow\infty$,  it holds that
\begin{align*}
  m_{q,a} \leq&\mathcal{I}_q(u)\\
= &\mathcal{I}_q(u)-\frac{1}{2^*}  P_q(u)\\
=&\Big(\frac{1}{2}-\frac{1}{2^*}\Big)\int_{\mathbb{R}^N}|\nabla u|^2dx+\frac{1}{2} \int_{\mathbb{R}^N}\Big[V(x)+\frac{1}{2^*}\langle\nabla V,x\rangle\Big]|u|^2dx+\mu\Big(\frac{\gamma_q}{2^*}
-\frac{1}{q}\Big)\int_{\mathbb{R}^N}|u|^{q}dx\\
 \leq& \liminf_{n\rightarrow\infty} \left( \Big(\frac{1}{2}-\frac{1}{2^*}\Big)\int_{\mathbb{R}^N}|\nabla u_n|^2dx+\frac{1}{2} \int_{\mathbb{R}^N}\Big[V(x)+\frac{1}{2^*}\langle\nabla V,x\rangle\Big]|u_n|^2dx
 +\mu\Big(\frac{\gamma_q}{2^*}
-\frac{1}{q}\Big)\int_{\mathbb{R}^N}|u_n|^{q}dx \right)\\
 =&\liminf_{n\rightarrow\infty} \Big(\mathcal{I}_q(u_n)-\frac{1}{2^*}P_q(u_n) \Big)\\
  \leq &\lim_{n\rightarrow\infty}\mathcal{I}_q(u_n)\\
   =&m_{q,a},
\end{align*}
which shows that $ \mathcal{I}_q(u)= m_{q,a}$.

Next,  let
\begin{align} \label{h91}
 U_a:=\big\{ u\in\mathcal{P}_{q,a}:\ \mathcal{I}_q(u)=m_{q,a}=\inf_{\mathcal{P}_{q,a} }\mathcal{I}_q  \big\}.
\end{align}
It is easy to see  from  Lemma \ref{Lem4.1}  that $  U_a= Z_a$.
We claim that
\begin{align}\label{h90}
u\in Z_a \ \Rightarrow \ |u|\in Z_a, \ \ |\nabla |u||_2=|\nabla u|_2.
\end{align}
Indeed, observe that $|\nabla |u||_2\leq|\nabla u|_2 $, then $\mathcal{I}_q(|u|)\leq \mathcal{I}_q(u)$ and $P_q(|u|)\leq P_q(u)=0$. By Lemma \ref{Lem2.3}, there exists $s_{|u|}\leq 0$ satisfying $s_{|u|}\ast |u|\in \mathcal{P}_{q,a}$. Hence, one has
\begin{align*}
 m_{q,a}&\leq \mathcal{I}_q(s_{|u|}\ast |u|) \\
 &=\frac{e^{2s_{|u|}}}{2}\int_{\mathbb{R}^N}|\nabla |u||^2dx+\frac{1}{2} \int_{\mathbb{R}^N} V(e^{-s_{|u|}}x)|u|^2dx-\frac{e^{2^*s_{|u|}}}{2^*} \int_{\mathbb{R}^N}|u|^{2^*}dx
-\frac{\mu}{q}e^{(\frac{q}{2}-1)Ns_{|u|}}\int_{\mathbb{R}^N}|u|^{q}dx\\
&\leq \frac{e^{2s_{|u|}}}{2}\int_{\mathbb{R}^N}|\nabla u|^2dx+\frac{1}{2} \int_{\mathbb{R}^N} V(e^{-s_{|u|}}x)|u|^2dx-\frac{e^{2^*s_{|u|}}}{2^*} \int_{\mathbb{R}^N}|u|^{2^*}dx
-\frac{\mu}{q}e^{(\frac{q}{2}-1)Ns_{|u|}}\int_{\mathbb{R}^N}|u|^{q}dx\\
&= \mathcal{I}_q(s_{|u|}\ast u)\\
&\leq  \mathcal{I}_q(  u)\\
&= m_{q,a},
\end{align*}
which shows that $|\nabla |u||_2=|\nabla u|_2 $. Thus we have $\mathcal{I}_q(  |u|)=\mathcal{I}_q(  u)=  m_{q,a} $ and $ P_q(|u|)=P_q(u)=0$. This gives that $|u|\in Z_a$.   Thus, $|u|$ is a non-negative real-valued solution to Eq. \eqref{h1} for some $\lambda \in \mathbb{R}^+$.
Let $u:=|u|$,
  we claim that $u>0$ on $\mathbb{R}^N$.
Indeed,  from Lemma \ref{Lem4.3}, we know that
\begin{align*}
-\Delta u= -\lambda u-V(x)u + |u|^{2^*-2}u+ |u|^{q-2}u\in L_{loc}^p(\mathbb{R}^N,  \mathbb{R}), \quad\quad \text{for any}\ p\geq 2,
\end{align*}
which and  Cald\'{e}ron-Zygmund inequality\cite{2001book} imply $u\in W_{loc}^{2, p}(\mathbb{R}^N, \mathbb{R} )$.  Moreover, by Sobolev embedding theorem, we have $u\in C_{loc}^{1, \alpha}(\mathbb{R}^N, \mathbb{R} )$   for $0 < \alpha< 1$. Then the strong maximum principle shows that $u  > 0$. Hence, we obtain that there exists a couple $(\lambda, u)\in \mathbb{R}^+\times E$ solving Eq. \eqref{h1}, where $u$ is a positive real-valued function and $\mathcal{I}_q(u)=  m_{q,a}$.
Therefore, we complete the  proof of  Theorem \ref{TH1}.  $\hfill\Box$\\

\section{ Proof of Theorem \ref{TH2}}\label{sec4}
In this section, we consider  the $L^2$-critical perturbation case, i.e., $q=2+\frac{4}{N}$, and prove Theorem \ref{TH2}.   Below we first present some preliminary lemmas that are necessary to prove Theorem \ref{TH2}.

 \begin{lemma}\label{Lem2.1j}
Let $(\widetilde{V}_2)$ be satisfied and $q=\overline{q} $.  Setting $\mu a^{\frac{4}{N}}< (\overline{a}_N)^{\frac{4}{N}}(1-\widetilde{\sigma}_2)$, for any $u\in\mathcal{P}_{\overline{q},a} $, there exists $\widetilde{\delta}>0$  such that $|\nabla u|_2\geq \widetilde{\delta}$.
  \end{lemma}
 \begin{proof} When  $q=\overline{q} $,
 we know that  $ \overline{q}\gamma_{\overline{q}}=2$.  Thus,  by $(\widetilde{V}_2)$ and \eqref{h10} and \eqref{h11}, similar to \eqref{h12},  for any $u\in\mathcal{P}_{\overline{q},a} $ one shows that
$$
(1-\widetilde{\sigma}_2-\mu \gamma_{\overline{q}} C_{N,\overline{q}}^{\overline{q}}a^{\overline{q}(1-\gamma_{\overline{q}})} )|\nabla u|_2^2\leq \mathcal{S}^{-\frac{N}{N-2}}|\nabla u|_2^{2^*}.
$$
By \eqref{h46}, let $\mu a^{\frac{4}{N}}
 < (\overline{a}_N)^{\frac{4}{N}}(1-\widetilde{\sigma}_2)$, we infer that  there exists $\widetilde{\delta}>0$ satisfying $|\nabla u|_2\geq\widetilde{\delta} $. Thus, we complete the proof of this lemma.
  \end{proof}

  \begin{lemma}\label{Lem2.4j}
  Assume that    $(\widetilde{V}_1)$ and $(\widetilde{V}_2)$ hold. Let $q=\overline{q}$ and $\mu  a^{\frac{4}{N}} < \left( 1 -\frac{N \widetilde{\sigma}_1}{2}-\frac{N -2 }{2} \widetilde{\sigma}_2  \right)  (\overline{a}_N)^{\frac{4}{N}}  $, then
$m_{\overline{q},a}=\inf_{u\in\mathcal{P}_{\overline{q},a}} \mathcal{I}_{\overline{q}}(u)>0$.
 \end{lemma}
  \begin{proof}
 From $(\widetilde{V}_1), (\widetilde{V}_2)$, Gagliardo-Nirenberg inequality \eqref{h10} and $\eqref{h6}$,    for any $u\in\mathcal{P}_{\overline{q},a} $ we deduce
\begin{align} \label{h17}
\mathcal{I}_{\overline{q}}(u)&= \mathcal{I}_{\overline{q}}(u)- \frac{1}{2^*} P_{ \overline{q}}(u)\nonumber\\
&=\frac{1}{N}\int_{\mathbb{R}^N}|\nabla u|^2dx+\frac{1}{2}\int_{\mathbb{R}^N} V(x)u^2dx+\frac{1}{22^*}\int_{\mathbb{R}^N} \langle \nabla V(x),x\rangle |u|^2dx+\left(\frac{\gamma_{\overline{q}}}{2^*}-\frac{1}{\overline{q}} \right)\mu \int_{\mathbb{R}^N} |u|^{\overline{q}}dx\nonumber\\
&\geq \frac{1}{N} \int_{\mathbb{R}^N}|\nabla u|^2dx-\frac{\widetilde{\sigma}_1}{2}\int_{\mathbb{R}^N}|\nabla u|^2dx- \frac{\widetilde{\sigma}_2}{2^*}\int_{\mathbb{R}^N}|\nabla u|^2dx-\frac{\mu}{N+2} \int_{\mathbb{R}^N} |u|^{\overline{q}}dx\nonumber\\
 &\geq \left( \frac{1}{N} -\frac{\widetilde{\sigma}_1}{2}- \frac{\widetilde{\sigma}_2}{2^*} \right)\int_{\mathbb{R}^N}|\nabla u|^2dx- \frac{\mu}{N+2} C_{N,\overline{q}}^{\overline{q}}a^{\frac{4}{N}} \int_{\mathbb{R}^N}|\nabla u|^2dx\nonumber\\
 &\geq \left( \frac{1}{N} -\frac{\widetilde{\sigma}_1}{2}- \frac{\widetilde{\sigma}_2}{2^*} - \frac{\mu}{N+2} C_{N,\overline{q}}^{\overline{q}}a^{\frac{4}{N}} \right) \int_{\mathbb{R}^N}|\nabla u|^2dx,
\end{align}
which and Lemma \ref{Lem2.1j} imply that  $m_{\overline{q},a}=\inf_{u\in\mathcal{P}_{\overline{q},a}} \mathcal{I}_{\overline{q}}(u)>0$ because $\widetilde{\sigma}_1<\frac{2}{N}$ and $\widetilde{\sigma}_2<\frac{2}{N-2}-\frac{N}{N-2} \widetilde{\sigma}_1 $  and $\mu  a^{\frac{4}{N}} < \left( 1 -\frac{N \widetilde{\sigma}_1}{2}-\frac{N -2 }{2} \widetilde{\sigma}_2  \right)  (\overline{a}_N)^{\frac{4}{N}}  $. Thus, we complete the proof.
 \end{proof}

 \vspace{0.2cm}
Consider the decomposition of $\mathcal{P}_{q,a}$ into the disjoint union
$$
\mathcal{P}_{q,a}=\left(\mathcal{P}_{q,a} \right)^+ \cup \left(\mathcal{P}_{q,a} \right)^0\cup \left(\mathcal{P}_{q,a} \right)^-,
$$
 where
 $$
 \left(\mathcal{P}_{q,a} \right)^{+(resp.\ 0,-)}=\left\{ u\in\mathcal{P}_{q,a}: \psi_u''(0)> (resp. \ =,<)0 \right\}.
 $$

\begin{lemma}\label{Lem2.2j}
  Assume that   $(\widetilde{V}_2)-(\widetilde{V}_3)$ hold and $q=\overline{q}$.  Let $ \mu a^{\frac{4}{N}} < (1-\frac{N}{2} \widetilde{\sigma}_2 -\frac{N-2}{4}\widetilde{\sigma}_3)(\overline{a}_N)^{\frac{4}{N}} $, then $\mathcal{P}_{\overline{q},a}= (\mathcal{P}_{\overline{q},a})^-$.
\end{lemma}
  \begin{proof}
By   \eqref{h13} and   \eqref{hjh13}  with $q=\overline{q}$,  using   Gagliardo-Nirenberg inequality \eqref{h10} and  $(\widetilde{V}_2)-(\widetilde{V}_3)$, for any $u\in \mathcal{P}_{\overline{q},a}$, one has
 \begin{align*}
\psi_u''(0)&=\psi_u''(0)-\frac{2N}{N-2}P_{\overline{q}}(u)\\
&=-\frac{4}{N-2} \int_{\mathbb{R}^N}|\nabla u|^2dx+\int_{\mathbb{R}^N} \langle \nabla W(x), x\rangle |u|^2dx+\frac{2N}{N-2}\int_{\mathbb{R}^N} W(x)  |u|^2dx+\mu  \frac{4\gamma_q}{N-2} \int_{\mathbb{R}^N}   |u|^{\overline{q}}dx\\
 &\leq \Big(\widetilde{\sigma}_3+  \frac{2N}{N-2}\widetilde{\sigma}_2 -\frac{4}{N-2}\Big)\int_{\mathbb{R}^N}|\nabla u|^2dx +\mu  \frac{4\gamma_q}{N-2} \int_{\mathbb{R}^N}   |u|^{\overline{q}}dx\\
 &\leq \Big(\widetilde{\sigma}_3+  \frac{2N}{N-2}\widetilde{\sigma}_2 -\frac{4}{N-2}+\mu a^{\frac{4}{N}} \frac{4}{N-2}  (\overline{a}_N)^{-\frac{4}{N}}\Big)\int_{\mathbb{R}^N}|\nabla u|^2dx\\
&<0,
\end{align*}
which shows that $u\in (\mathcal{P}_{\overline{q},a})^-$. Hence, $\mathcal{P}_{\overline{q},a}= (\mathcal{P}_{\overline{q},a})^-$. Thus, we complete the proof of this lemma.
 \end{proof}

 \begin{lemma}\label{Lem2.3j}
  Assume that $(\widetilde{V}_2)$ and  $(\widetilde{V}_3)$  hold and $q=\overline{q}$. Let
  $$ \mu a^{\frac{4}{N}} < \min\left\{ 1-\widetilde{\sigma}_2,  \frac{2}{N} -\widetilde{\sigma}_2-\frac{\widetilde{\sigma}_3}{2^*} \right\}(\overline{a}_N)^{\frac{4}{N}},
   $$
  for any $u\in S_a$, there is a unique $s_u>0$ such that $ s_u\ast u \in  \mathcal{P}_{\overline{q},a} $. Moreover, $\mathcal{I}_{\overline{q}}(s_u\ast u )=\max_{s>0}\mathcal{I}_{\overline{q}}(s\ast u )$, and   $s_u<0$ if and only  if $P_{\overline{q}}(u)<0$, and  $ \psi_u$   is strictly decreasing and concave on  $(s_u, +\infty)$. And the map $u\in S_a\mapsto s_u\in \mathbb{R}$ is of class $C^1$.
\end{lemma}

 \begin{proof}
Since $ \psi_u'(s)= P_{\overline{q}}(s\ast u )$, then we only  prove that $\psi_u'(s)$ has a unique root in $\mathbb{R}$. By \eqref{h8} and $(\widetilde{V}_2)$, for any $u\in S_a$,
\begin{align} \label{h49}
\psi_u'(s)&=
e^{2s} \int_{\mathbb{R}^N}|\nabla u|^2dx-\int_{\mathbb{R}^N} W(e^{-s}x)|u|^2dx-e^{2^*s} \int_{\mathbb{R}^N}|u|^{2^*}dx-\mu \gamma_q e^{(\frac{q}{2}-1)Ns}\int_{\mathbb{R}^N}|u|^{q}dx\nonumber\\
&\geq \left(1-\widetilde{\sigma}_2-\mu \gamma_{\overline{q}} C_{N, \overline{q}}^{\overline{q}}a^{\frac{4}{N}}\right) e^{2s} |\nabla u|_2^2-e^{2^*s} \int_{\mathbb{R}^N}|u|^{2^*}dx.
 \end{align}
 Since $\mu a^{\frac{4}{N}}<   (1-\widetilde{\sigma}_2)\overline{a}_N^{\frac{4}{N}} $ with $\widetilde{\sigma}_2<1$, we have $1-\widetilde{\sigma}_2-\mu \gamma_{\overline{q}} C_{N, \overline{q}}^{\overline{q}}a^{\frac{4}{N}}>0$.
   Therefore, from \eqref{h49},  it holds that $ \psi_u'(s)>0$ for $s\rightarrow-\infty$. Hence there exists $s_0\in \mathbb{R}$ such that $ \psi_u(s)$ is increasing on $(-\infty, s_0 )$. In addition, we know that $\psi_u(s)\rightarrow -\infty $ as $s\rightarrow+\infty$. Thus, there is $s_1\in \mathbb{R}$ with $s_1>s_0$ such that
 \begin{align}\label{h36j}
  \psi_u(s_1)=\max_{s>0} \psi_u(s).
  \end{align}
   Therefore, $ \psi_u'(s_1)=0$, namely, $s_1\ast u\in \mathcal{P}_{q,a}$. Assume that there exists $s_2\in \mathbb{R}$ satisfying   $s_2\ast u\in \mathcal{P}_{q,a}$.   Without loss of generality, assume that $ s_2>s_1$.  From Lemma \ref{Lem2.2}, we have $ \psi_u''(s_2)<0$. Hence, there is $s_3\in(s_1, s_2)$ such that
  $$
  \psi_u(s_3)=\min_{s\in(s_1, s_2) } \psi_u(s).
  $$
  Thus, we deduce that $ \psi_u''(s_3)\geq0$ and $ \psi_u'(s_3)=0$, which implies that $s_3\ast u\in \mathcal{P}_{q,a}$ and      $s_3\ast u\in (\mathcal{P}_{q,a})^+\cup (\mathcal{P}_{q,a})^0$. This is a contradiction. Hence,  setting $s_u=s_1$,  it follows from \eqref{h36j} that $\mathcal{I}_q(s_u\ast u )=\max_{s>0}\mathcal{I}_q(s\ast u ) $. Furthermore,
  \begin{align} \label{h47j}
  \psi_u'(s)<0  \quad \Leftrightarrow \quad  s>s_u.
\end{align}
   Hence,   $ \psi_u'(0)=P_{q}(u)<0$ if and only if $s_u<0$.  In the following, we show that   $ \psi_u$   is strictly decreasing and concave on  $(s_u, +\infty)$. Clearly, it follows from \eqref{h47j} that  $ \psi_u$   is strictly decreasing  on  $(s_u, +\infty)$. Now, we   verify that $\psi_u''(s)<0$ on  $(s_u, +\infty)$.   From $(\widetilde{V}_3)$,
     \begin{align} \label{h48j}
   \psi_u''(s)=&  2e^{2s} \int_{\mathbb{R}^N}|\nabla u|^2dx+\int_{\mathbb{R}^N} \langle \nabla W(e^{-s}x), e^{-s}x\rangle |u|^2dx-2^*e^{2^*s}\int_{\mathbb{R}^N}|u|^{2^*}dx-\mu q\gamma_q^2  e^{q\gamma_q s}\int_{\mathbb{R}^N}|u|^{q}dx\nonumber\\
    \leq &  (2+\widetilde{\sigma}_3)e^{2s} \int_{\mathbb{R}^N}|\nabla u|^2dx -2^*e^{2^*s}\int_{\mathbb{R}^N}|u|^{2^*}dx.
   \end{align}
Let
$$
g_1(s)=\left(1-\widetilde{\sigma}_2-\mu \gamma_{\overline{q}} C_{N, \overline{q}}^{\overline{q}}a^{\frac{4}{N}}\right) e^{2s} |\nabla u|_2^2-e^{2^*s} \int_{\mathbb{R}^N}|u|^{2^*}dx
$$
and
$$
g_2(s)=(2+\widetilde{\sigma}_3)e^{2s} \int_{\mathbb{R}^N}|\nabla u|^2dx -2^*e^{2^*s}\int_{\mathbb{R}^N}|u|^{2^*}dx.
$$
Obviously, $g_1$ and $g_2$ has a unique zero point. Assume that $g_1(\widetilde{s}_1)=0$ and  $g_2(\widetilde{s}_2)=0$.  It holds that $g_1(s)>0 $   on $ (-\infty, \widetilde{s}_1)$ and $g_1(s)<0 $ on $ ( \widetilde{s}_1, +\infty)$ and $g_2(s)>0 $  on $ (-\infty, \widetilde{s}_2)$ and $g_2(s)<0 $ on $ ( \widetilde{s}_2, +\infty)$.     From above analysis,  we know from \eqref{h49} that $\psi_u $ is strictly  increasing on $ (-\infty, \widetilde{s}_1)$.
Moreover, since $\psi_u $ is strictly  increasing on $ (-\infty, s_u)$ and strictly  decreasing on $ ( s_u, +\infty)$,  we have $\widetilde{s}_1\leq  s_u$.  If $ \widetilde{s}_2\leq s_u$, then we have $\psi_u''(s)<0$ on  $(s_u, +\infty)$. Thus, we only need to show that $\widetilde{s}_2\leq \widetilde{s}_1 $,  which is equivalent to verify  $ g_2(\widetilde{s}_1)\leq 0$. Since $g_1(\widetilde{s}_1 )=0$, namely,
$$
\left(1-\widetilde{\sigma}_2-\mu \gamma_{\overline{q}} C_{N, \overline{q}}^{\overline{q}}a^{\frac{4}{N}}\right) e^{2\widetilde{s}_1} |\nabla u|_2^2=e^{2^*\widetilde{s}_1} \int_{\mathbb{R}^N}|u|^{2^*}dx,
$$
then
 \begin{align} \label{h50j}
g_2(\widetilde{s}_1)&=(2+\widetilde{\sigma}_3)e^{2\widetilde{s}_1} \int_{\mathbb{R}^N}|\nabla u|^2dx -2^*e^{2^*\widetilde{s}_1}\int_{\mathbb{R}^N}|u|^{2^*}dx \nonumber\\
&=(2+\widetilde{\sigma}_3-2^* (1-\widetilde{\sigma}_2-\mu \gamma_{\overline{q}} C_{N, \overline{q}}^{\overline{q}}a^{\frac{4}{N}} ))e^{2\widetilde{s}_1} \int_{\mathbb{R}^N}|\nabla u|^2dx\nonumber\\
&<0
 \end{align}
because $\mu a^{\frac{4}{N}}<(\frac{2}{N} -\widetilde{\sigma}_2-\frac{\widetilde{\sigma}_3}{2^*}) \overline{a}_N^{\frac{4}{N}} $.
  So, we obtain that $ \psi_u$   is    strictly     concave on  $(s_u, +\infty)$. Similar to Lemma \ref{Lem2.3}, we show that the map $u\in S_a\mapsto s_u\in \mathbb{R}$ is of class $C^1$.
 Thus, we complete the proof.
  \end{proof}

  \begin{lemma}\label{Lem2.5j}
 Assume that $(\widetilde{V}_1)-(\widetilde{V}_2)$ hold  and $q=\overline{q}$. Let
 $$
 \mu a^{\frac{4}{N}}<\min\left\{\Big( 1 -\frac{N \widetilde{\sigma}_1}{2}-\frac{N -2 }{2} \widetilde{\sigma}_2  \Big),  1-\widetilde{\sigma}_2, \frac{N+2}{N}- \Big(2+\frac{4}{N}\Big)\widetilde{\sigma}_1 \right\} (\overline{a}_N)^{\frac{4}{N}},
 $$
  then there is a constant $k>0$ small enough such that
 $$
 0<\sup_{\overline{A_k}} \mathcal{I}_{\overline{q}}<m_{\overline{q},a} \quad \quad \text{and}\quad \quad u\in \overline{A_k}\Rightarrow\mathcal{I}_{\overline{q}}(u), P_{\overline{q}}(u)>0,
 $$
 where $A_k=\{u\in S_a:|\nabla u|_2^2<k\}$.
 \end{lemma}
  \begin{proof}
 Firstly, by $(\widetilde{V}_1)$, similar to \eqref{h18},  setting $\mu a^{\frac{4}{N}}<(\frac{1}{2}-\widetilde{\sigma}_1)(2+\frac{4}{N}) (\overline{a}_N)^{\frac{4}{N}} $,   we have $\mathcal{I}_{\overline{q}}(u)>0$ for any $u\in \overline{A}_k$ with $k>0$ small enough. Then,   similar to \eqref{h19}, by $(\widetilde{V}_2)$ and the fact that $\mu a^{\frac{4}{N}}<(1-\widetilde{\sigma}_2)  (\overline{a}_N)^{\frac{4}{N}} $, one gets $P_{\overline{q}}(u)>0$ for any $u\in \overline{A}_k$ with $k>0$ small enough. Furthermore, choosing $k$ sufficient small,  by Lemma \ref{Lem2.4j},
$$
\mathcal{I}_{\overline{q}}(u)\leq \frac{1}{2}\int_{\mathbb{R}^N}|\nabla u|^2dx < m_{\overline{q},a}
$$
for all $u\in \overline{A}_k$. Hence, we complete the proof of this lemma.
   \end{proof}

 {\bf {Proof of Theorem \ref{TH2}}}:   Let
 $$0<\mu a^{\frac{4}{N}}< \min\left\{1-\widetilde{\sigma}_2, \frac{N+2}{N}- \big(2+\frac{4}{N}\big)\widetilde{\sigma}_1,    1 -\frac{N \widetilde{\sigma}_1}{2}-\frac{N -2 }{2} \widetilde{\sigma}_2,
 \frac{2}{N}-\widetilde{\sigma}_2-\frac{\widetilde{\sigma}_3}{2^*}     \right\}(\overline{a}_N)^{\frac{4}{N}},
 $$
  by  Lemmas   \ref{Lem2.9}, \ref{Lem4.1},   \ref{Lem2.1j},  \ref{Lem2.4j}, \ref{Lem2.2j},  \ref{Lem2.3j},    \ref{Lem2.5j}  and \ref{Lem4.3}, and using  Proposition \ref{TH3},  similar to the proof of Theorem \ref{TH1},  we   infer that there exists a couple  $(\lambda, u) \in \mathbb{R}^+ \times E$   solving Eq. \eqref{h1}, where $u$ is a real-valued positive function in $\mathbb{ R}^N$ and $\mathcal{I}_{\overline{q}}(u)= m_{\overline{q},a} $.   So  we complete the proof of   Theorem \ref{TH2}.   $\hfill\Box$

\section{ Proof of Theorem \ref{TH6}}\label{sec6}
In this section, we study  the $L^2$-subcritical  perturbation case, i.e., $2<q<\overline{q}:=2+\frac{4}{N}$.  Firstly,  we give some preliminary lemmas for proving Theorem \ref{TH6}.

For any $u\in E$, we obtain from $(\widehat{V}_1)$, \eqref{h10}  and \eqref{h11}   that
 \begin{align}\label{h73}
 \mathcal{I}_{q}(u)=&\frac{1}{2}\int_{\mathbb{R}^N}|\nabla u|^2+V(x)|u|^2dx
-\frac{1}{2^*}\int_{\mathbb{R}^N}|u|^{2^*}dx-\frac{\mu}{q}\int_{\mathbb{R}^N}|u|^{q}dx\nonumber\\
 \geq&\frac{1}{2} (1-\widehat{\sigma}_1)|\nabla u|_2^2-\frac{1}{2^*}\mathcal{S}^{-\frac{2^*}{2}}|\nabla u|_2^{2^*}-\frac{\mu}{q}C_{N,q}^q |u|_2^{q(1-\gamma_q)}|\nabla u|_2^{q\gamma_q}\nonumber\\
 \geq& |\nabla u|_2^2\left(\frac{1}{2} (1-\widehat{\sigma}_1)-\frac{1}{2^*}\mathcal{S}^{-\frac{2^*}{2}}|\nabla u|_2^{2^*-2} -  \frac{\mu}{q}C_{N,q}^q |u|_2^{q(1-\gamma_q)}|\nabla u|_2^{q\gamma_q-2}  \right),
\end{align}
where $q\gamma_q<2 $.
Now we consider the function $ f(a, k)$ defined  by
$$
f(a, k)=\frac{1}{2} (1-\widehat{\sigma}_1)-\frac{1}{2^*}\mathcal{S}^{-\frac{2^*}{2}}k^{2^*-2} -  \frac{\mu}{q}C_{N,q}^q a^{q(1-\gamma_q)}k^{q\gamma_q-2}\quad\quad \text{ for all} \ (a, k)\in \mathbb{R}^+\times \mathbb{R}^+.
$$
In the following,  for any $\mu>0$ fixed, we have

 \begin{lemma}\label{Lem5.1}
 Assume that $( \widehat{V}_1)$ holds  and $2<q<\overline{q}$.   For any $a>0$ fixed, then the function $ f(a, k)$ has a unique global maximum respect to $k $ and the maximum value satisfies
  \begin{align*}
 \begin{cases}
  \max_{k>0}  f(a, k)>0 \quad\quad \text{if}\ a<a_0,\\
  \max_{k>0}  f(a, k)=0 \quad\quad \text{if}\ a=a_0,\\
  \max_{k>0}  f(a, k)<0 \quad\quad \text{if}\ a>a_0,
  \end{cases}
   \end{align*}
where
 \begin{align}\label{h74}
a_0:= \left(\frac{1}{2K} (1-\widehat{\sigma}_1) \right)^{ \frac{ 2^*-q\gamma_q }{q(1-\gamma_q)( 2^*-2) }}
\end{align}
 with
 $$
 K= \frac{1}{2^*}\mathcal{S}^{-\frac{2^*}{2}}  \Big(  \frac{\mu 2^* (2-q\gamma_q)C_{N,q}^q \mathcal{S}^{\frac{2^*}{2}}  }{q(2^*-2)}\Big)^{\frac{2^*-2}{2^*-q\gamma_q}}- \frac{\mu}{q}C_{N,q}^q   \Big(  \frac{\mu 2^* (2-q\gamma_q)C_{N,q}^q \mathcal{S}^{\frac{2^*}{2}}  }{q(2^*-2)}\Big)^{\frac{q\gamma_q-2}{2^*-q\gamma_q}}.
 $$
 \end{lemma}
  \begin{proof}
 By definition of $f(a, k)$, we have
 $$
  f'_k(a,k)= -\frac{2^*-2}{2^*} \mathcal{S}^{-\frac{2^*}{2}}k^{2^*-3}- \frac{\mu}{q}C_{N,q}^q(q\gamma_q-2) a^{q(1-\gamma_q)}k^{q\gamma_q-3}.
 $$
  Hence, the equation $f'_k(a,k) =0$ has a unique solution given by
 \begin{align}\label{h75}
  k_a=\left(  \frac{\mu 2^* (2-q\gamma_q)C_{N,q}^q \mathcal{S}^{\frac{2^*}{2}}  }{q(2^*-2)}\right)^{\frac{1}{2^*-q\gamma_q}} a^{ \frac{q(1-\gamma_q)}{ 2^*-q\gamma_q}}>0.
\end{align}
 Taking into account that  $f(a,k)\rightarrow-\infty$  as $ k\rightarrow0$ and  $f(a,k)\rightarrow-\infty$  as $ k\rightarrow +\infty$, we obtain that $ k_a$ is the unique global maximum point of the function  $f(a, k)$ with any fixed $a>0$ and the maximum value is, for any $a>0$ fixed,
  \begin{align*}
 \max_{k>0}  f(a, k)= f(a, k_a) =&\frac{1}{2} (1-\widehat{\sigma}_1)-\frac{1}{2^*}\mathcal{S}^{-\frac{2^*}{2}}k_a^{2^*-2} -  \frac{\mu}{q}C_{N,q}^q a^{q(1-\gamma_q)}k_a^{q\gamma_q-2}\\
 =& \frac{1}{2} (1-\widehat{\sigma}_1)-\frac{1}{2^*}\mathcal{S}^{-\frac{2^*}{2}}  \Big(  \frac{\mu 2^* (2-q\gamma_q)C_{N,q}^q \mathcal{S}^{\frac{2^*}{2}}  }{q(2^*-2)}\Big)^{\frac{2^*-2}{2^*-q\gamma_q}} a^{ \frac{q(1-\gamma_q)( 2^*-2)}{ 2^*-q\gamma_q}}\\
  &-  \frac{\mu}{q}C_{N,q}^q a^{q(1-\gamma_q)}\Big(  \frac{\mu 2^* (2-q\gamma_q)C_{N,q}^q \mathcal{S}^{\frac{2^*}{2}}  }{q(2^*-2)}\Big)^{\frac{q\gamma_q-2}{2^*-q\gamma_q}} a^{ \frac{q(1-\gamma_q)(q\gamma_q-2)}{ 2^*-q\gamma_q}}\\
  =&\frac{1}{2} (1-\widehat{\sigma}_1) - K a^{ \frac{q(1-\gamma_q)( 2^*-2)}{ 2^*-q\gamma_q}},
\end{align*}
 where
 $$
 K= \frac{1}{2^*}\mathcal{S}^{-\frac{2^*}{2}}  \Big(  \frac{\mu 2^* (2-q\gamma_q)C_{N,q}^q \mathcal{S}^{\frac{2^*}{2}}  }{q(2^*-2)}\Big)^{\frac{2^*-2}{2^*-q\gamma_q}}- \frac{\mu}{q}C_{N,q}^q   \Big(  \frac{\mu 2^* (2-q\gamma_q)C_{N,q}^q \mathcal{S}^{\frac{2^*}{2}}  }{q(2^*-2)}\Big)^{\frac{q\gamma_q-2}{2^*-q\gamma_q}}.
 $$
 By the definition of  $a_0$, we deduce that $\max_{k>0}  f(a_0, k) =0$, and hence the lemma follows.
  \end{proof}

\begin{lemma}\label{Lem5.2}
Assume that $( \widehat{V}_1)$ holds  and $2<q<\overline{q}$.   Let $(a_1, k_{1})\in \mathbb{R}^+\times \mathbb{R}^+$ be such that $f(a_1, k_{1}) \geq0$. Then for any $a_2\in(0, a_1]$, there hold that
$$
f(a_2, k_{2})\geq 0 \quad\quad \text{if}\ k_{2}\in\Big[\frac{a_2}{a_1} k_{1},  k_{1}\Big].
$$
 \end{lemma}
  \begin{proof}
 Since $a \rightarrow f(\cdot, k)$ is a non-increasing function we clearly have that, for any $a_2\in(0, a_1]$,
\begin{align}\label{h76}
   f(a_2, k_{1})\geq f(a_1, k_{1})\geq0.
  \end{align}
By the definition of  $f(a,k)$ and by direct calculations, we get that
\begin{align}\label{h77}
     f(a_2,  \frac{a_2}{a_1}  k_{1})\geq f(a_1, k_{1})\geq0.
   \end{align}
  We observe that if $f(a_2, k')\geq 0$ and $ f(a_2, k'' )\geq 0$,  then
  $$
  f(a_2, k)\geq 0 \quad\quad \text{for any}\ k\in [k',  k''].
  $$
Indeed, if   $f(a_2, k)<0$ for some $k \in [k',  k'']$, then there exists a local minimum point on  $(k',  k'')$ and this contradicts the fact that the function $f(a_2, k)$ has a unique   global maximum by Lemma \ref{Lem5.1}. Hence, by \eqref{h76} and \eqref{h77}, taking $k'= \frac{a_2}{a_1} k_{1}$ and $k''= k_{1}$, we get the conclusion.
  \end{proof}

Now let $ a_0>0$ be given by \eqref{h74} and $k_0:=k_{a_0}>0$ being determined by \eqref{h75}. Note that by Lemmas \ref{Lem5.1} and  \ref{Lem5.2}, we have that $f(a_0, k_0) =0$ and $f(a, k_0) >0$ for all $a\in (0, a_0)$. We define
 $$
 D_{k_0}:=\{u\in E: |\nabla u|_2^2<k_0^2 \} \quad\quad \text{and} \quad\quad V_a:=S_a\cap  D_{k_0}
$$
We shall now consider the following local minimization problem: for any $a \in(0, a_0)$,
\begin{align}\label{h78}
m(a):=\inf_{u\in V_a} \mathcal{I}_q(u).
\end{align}

\begin{lemma}\label{Lem5.3}
Assume that $( \widehat{V}_1)$ holds  and $2<q<\overline{q}$.    For any $a\in(0, a_0)$, it holds that
$$
m(a)=\inf_{u\in V_a} \mathcal{I}_q(u)<0< \inf_{u\in \partial V_a} \mathcal{I}_q(u).
$$
 \end{lemma}
  \begin{proof} For any $u\in S_a$, since
 \begin{align*}
\mathcal{I}_q(s\ast u )
&=\frac{e^{2s}}{2}\int_{\mathbb{R}^N}|\nabla u|^2dx+\frac{1}{2} \int_{\mathbb{R}^N} V(e^{-s}x)|u|^2dx-\frac{e^{2^*s}}{2^*} \int_{\mathbb{R}^N}|u|^{2^*}dx
-\frac{\mu}{q}e^{(\frac{q}{2}-1)Ns}\int_{\mathbb{R}^N}|u|^{q}dx\nonumber\\
&\leq e^{2s}\left(\big(\frac{1}{2}+\widehat{\sigma}_1\big)|\nabla u|_2^2- \frac{e^{(2^*-2)s}}{2^*} \int_{\mathbb{R}^N}|u|^{2^*}dx-\frac{\mu}{q}e^{(q\gamma_q-2)s} \int_{\mathbb{R}^N}|u|^{q}dx  \right)
\end{align*}
and $q\gamma_q<2$, then there exists $s_0<<-1$ such that $|\nabla (s_0\ast u)|_2^2= e^{2 s_0} |\nabla u|_2^2<k_0^2$ and $ \mathcal{I}_q(s_0\ast u )<0$. This implies that $m(a)<0$. Moreover, in view of \eqref{h73} and the fact that $f(a, k_0) >f(a_0, k_0)=0$ for all $a\in (0, a_0)$,  we have  $\mathcal{I}_q(u)\geq |\nabla u|_2^2 f(|u|_2, |\nabla u|_2)=k_0^2f(a, k_0)>0$      for any $u\in  \partial V_a $.
Hence, we complete the proof.
  \end{proof}

\begin{lemma}\label{Lem5.7}
Assume that $( \widehat{V}_1)$ hold  and $2<q<\overline{q}$.  Let $\{u_n\}\subset D_{k_0}$
be  such that   $|u_n|_q\rightarrow 0$ as $n\rightarrow\infty$. Then there exists a  $\beta_0>0$   such that
$$
\mathcal{I}_q(u_n)\geq \beta_0|\nabla u_n|_2^2+o(1).
$$
 \end{lemma}
  \begin{proof}
Let $\{u_n\}\subset D_{k_0}$
be  such that   $|u_n|_q\rightarrow 0$,  we obtain from $(\widehat{V}_1)$, \eqref{h10} and \eqref{h11}  that
 \begin{align*}
 \mathcal{I}_{q}(u_n)=&\frac{1}{2}\int_{\mathbb{R}^N}|\nabla u_n|^2+V(x)|u_n|^2dx
-\frac{1}{2^*}\int_{\mathbb{R}^N}|u_n|^{2^*}dx+o(1) \\
 \geq&\frac{1}{2} (1-\widehat{\sigma}_1)|\nabla u_n|_2^2-\frac{1}{2^*}\mathcal{S}^{-\frac{2^*}{2}}|\nabla u_n|_2^{2^*}+o(1)  \\
 \geq& |\nabla u_n|_2^2\left(\frac{1}{2} (1-\widehat{\sigma}_1)-\frac{1}{2^*}\mathcal{S}^{-\frac{2^*}{2}}|\nabla u_n|_2^{2^*-2}    \right)+o(1) \\
  \geq&  |\nabla u_n|_2^2\left(\frac{1}{2} (1-\widehat{\sigma}_1)-\frac{1}{2^*}\mathcal{S}^{-\frac{2^*}{2}}k_0^{2^*-2}    \right)+o(1),
\end{align*}
which and $f(a_0, k_0)=0$ imply that
$$
\beta_0:= \frac{1}{2} (1-\widehat{\sigma}_1)-\frac{1}{2^*}\mathcal{S}^{-\frac{2^*}{2}}k_0^{2^*-2} = \frac{\mu}{q}C_{N,q}^q a_0^{q(1-\gamma_q)}k_0^{q\gamma_q-2}>0.
$$
Thus, we complete the proof.
  \end{proof}

Now, we collect some properties of $m(a)$ defined in \eqref{h78}.
\begin{lemma}\label{Lem5.4}
Assume that $( \widehat{V}_1)$ hold  and $2<q<\overline{q}$.  It holds that
 \begin{itemize}
   \item [$(i)$] $a\in(0,a_0)\mapsto m(a)$ is a continuous mapping.
   \item [$(ii)$] Let $a\in(0, a_0)$.  For any $\alpha\in (0, a)$, we have $ \frac{\alpha^2}{a^2}m(a)<m(\alpha)<0 $.
 \end{itemize}
 \end{lemma}
  \begin{proof}
Similarly as in \cite[Lemma 2.6]{2022-JMPA-jean}, by $( \widehat{V}_1)$ we can prove $(i)$. It remains to show $(ii)$. Let $\kappa>1 $  such that $a= \kappa \alpha$.   Let $\{u_n\}\subset V_{\alpha}$
be a minimizing sequence with respect to $m(\alpha)$, that is,
$$
 \mathcal{I}_{q}(u_n)\rightarrow m(\alpha) \quad\quad \text{as}\ n\rightarrow\infty.
$$
By Lemma \ref{Lem5.3}, we get $m(\alpha)<0 $. Letting $v_n=\kappa u_n$, we have  $|v_n|_2^2=a^2$. Moreover, since $ f(a, k_0)\geq f(a_0, k_0)\geq0$,     from Lemma  \ref{Lem5.2}  we deduce that
$$
f( \alpha, k)>0 \quad\quad \text{for}\ k\in \big[ \frac{\alpha}{a}k_0, k_0\big].
$$
Then,  by \eqref{h73} we get $f(\alpha, |\nabla u_n|_2)<0$, which implies that
$$
|\nabla u_n|_2< \frac{\alpha}{a}k_0.
$$
Then, one infers  $|\nabla v_n|_2^2=\kappa^2|\nabla u_n|_2^2 <k_0^2 $, which and  $|  v_n|_2^2=a^2$  show that $v_n\in V_a$. Thus,
$$
m(a)\leq \mathcal{I}_q(v_n)=\kappa^2 \mathcal{I}_q(u_n)+\frac{(\kappa^2 -\kappa^q )}{q}\mu \int_{\mathbb{R}^N}|u_n|^qdx+ \frac{(\kappa^2 -\kappa^{2^*} )}{2^*}  \int_{\mathbb{R}^N}|u_n|^{2^*}dx.
$$
Now, we claim that there exists a positive constant $C > 0$ and $n_0\in \mathbb{N}$ such that
 $$
 \int_{\mathbb{R}^N}|u_n|^qdx>C
 $$
 for all $n \geq n_0$.  Indeed, otherwise we have
   $|u_n|_q\rightarrow 0$ as $n\rightarrow\infty$. Then it follows from Lemma \ref{Lem5.7} that
   $$
   0>   m(\alpha) =\lim_{n\rightarrow\infty}\mathcal{I}_{q}(u_n) \geq \beta_0 \lim_{n\rightarrow\infty}|\nabla u_n|_2^2\geq 0,
   $$
 which is a contradiction. Hence  $  \int_{\mathbb{R}^N}|u_n|^qdx>C$  for all $n \geq n_0$. Then, in view of $\kappa>1$, one infers
\begin{align*}
m(a)\leq \mathcal{I}_q(v_n)=&\kappa^2 \mathcal{I}_q(u_n)+\frac{(\kappa^2 -\kappa^q )}{q}\mu \int_{\mathbb{R}^N}|u_n|^qdx+ \frac{(\kappa^2 -\kappa^{2^*} )}{2^*}  \int_{\mathbb{R}^N}|u_n|^{2^*}dx \\
\leq&\kappa^2 \mathcal{I}_q(u_n)+\frac{(\kappa^2 -\kappa^q )}{q}\mu C \\
\leq& \kappa^2 m(\alpha) +\frac{(\kappa^2 -\kappa^q )}{q}\mu C+o_n(1),
\end{align*}
 which shows that, as $n\rightarrow\infty$,
 $$
 m(a)< \kappa^2 m(\alpha)=\frac{a^2}{\alpha^2} m(\alpha).
 $$
 We complete the proof of this lemma.
  \end{proof}

\begin{lemma}\label{Lem5.5}
Assume that $( \widehat{V}_1)$ holds  and $q\in(2, \overline{q})$.  Let  $m^\infty(a ):= \inf_{u\in V_a} \mathcal{I}_q^\infty(u)$,
it  holds that $m(a )< m^\infty(a )$ for any $a\in(0, a_0)$.
 \end{lemma}
  \begin{proof}
According to \cite[ Theorem 1.2]{2022-JMPA-jean}, we know that there exists $0<u_0\in V_a$ satisfying $\mathcal{I}_q^\infty(u_0)=m^\infty(a )$ for any $a\in (0, a_0)$. Hence, since $V(x)\not\equiv 0$ and $\sup_{x\in\mathbb{R}^N} V(x) =0$,  we see that
$$
m(a )\leq \mathcal{I}_q(u_0)=\mathcal{I}_q^\infty(u_0)+\int_{\mathbb{R}^N}V(x)|u_0|^{2}dx< \mathcal{I}_q^\infty(u_0)=m^\infty(a ).
$$
Thus, we get $m(a )< m^\infty(a ) $ for any $a\in (0, a_0)$. The proof of this lemma is finished.
  \end{proof}

\begin{lemma}\label{Lem5.6}
Assume that $( \widehat{V}_1)$ hold  and $2<q<\overline{q}$.  Let $\{u_n\}\subset D_{k_0}$ and $|u_n|_2\rightarrow a$
be a minimizing sequence with respect to $m(a)$ with $u_n \rightharpoonup u$
in $E$, $u_n(x) \rightarrow u(x)$ a.e. in $\mathbb{R}^N$ and $u\neq0$. Then, $u\in V_a$, $ \mathcal{I}_q(u) = m(a ) $ and
$u_n  \rightarrow u$ in $E$.
 \end{lemma}
  \begin{proof}  Let $\{u_n\}\subset D_{k_0}$ and $|u_n|_2\rightarrow a$
be a minimizing sequence with respect to $m(a)$ with $u_n \rightharpoonup u$
in $E$, $u_n(x) \rightarrow u(x)$ a.e. in $\mathbb{R}^N$.
Setting $w_n:=u_n-u\rightharpoonup 0$ in $E$. The aim is to prove $w_n\rightarrow 0$ in $E$. By the Br\'{e}zis-Lieb lemma \cite{1993book}, we get
\begin{align}
|\nabla u_n|_2^2&=|\nabla w_n|_2^2+|\nabla u|_2^2+o(1);\label{h80}\\
|u_n|_2^2&=|w_n|_2^2+|u|_2^2+o(1); \label{h79}\\
|u_n|_t^t&=|w_n|_t^t+|u|_t^t+o(1),\nonumber
\end{align}
where $t\in(2, 2^*]$. Then we have
\begin{align}\label{h81}
\mathcal{I}_q(u_n)=\mathcal{I}_q(w_n)+\mathcal{I}_q(u)+o(1).
\end{align}
Now we claim that, as $n\rightarrow\infty$,
\begin{align}\label{h82}
|w_n|_2\rightarrow 0.
\end{align}
Indeed, let $|u|_2=a_1$. Since $u\neq 0$, \eqref{h80} and  \eqref{h79}, we have $a_1\in(0,a]$ and $u\in V_{a_1}$. If $a_1=a$, then, by \eqref{h79} we are done. If $a_1\in(0, a)$,  then $|w_n|_2^2<a^2$ and $|\nabla w_n|_2^2\leq |\nabla u_n|_2^2<k_0^2$ for $n$ large enough. Hence we know that
$w_n\in V_{|w_n|_2} $ and $\mathcal{I}_q(w_n)\geq m(|w_n|_2 ) $ for $n$ large enough. Hence from \eqref{h81}, Lemma \ref{Lem5.4} and the fact that $u\in V_{a_1}$ one has, for $n$ large enough,
\begin{align*}
m(a)+o(1)=\mathcal{I}_q(u_n)&=\mathcal{I}_q(w_n)+\mathcal{I}_q(u)+o(1)\\
&\geq m(|w_n|_2 )+ m(a_1)+o(1)\\
&\geq \frac{|w_n|_2^2}{a^2} m(a)+ m(a_1)+o(1).
\end{align*}
Then, let $n\rightarrow\infty$,
$$
m(a)\geq \lim_{n\rightarrow\infty}\frac{|w_n|_2^2}{a^2} m(a)+ m(a_1)>\lim_{n\rightarrow\infty} \frac{|w_n|_2^2}{a^2} m(a)+ \frac{a_1^2}{a^2} m(a)=m(a),
$$
which is a contradiction. Thus, $ |u|_2=a$ and then $u\in V_a$. This implies that \eqref{h82} holds. In what follows, we prove that
\begin{align}\label{h88}
  |\nabla w_n|_2\rightarrow 0 \quad\quad \text{ as} \ n\rightarrow\infty.
 \end{align}
By \eqref{h80},  $|\nabla w_n|_2^2\leq |\nabla u_n|_2^2<k_0^2$ for $n$ large enough. Thus, $\{w_n\}\subset D_{k_0}$ and $\{w_n\}$ is bounded in $E$  for $n$ large enough. In view of
Gagliardo-Nirenberg inequality  \eqref{h10} and \eqref{h82}, we deduce that
 \begin{align}\label{h89}
 |w_n|_t\rightarrow 0  \quad\quad \text{ as} \ n\rightarrow\infty,
 \end{align}
 where $t\in(2, 2^*)$.
 Then it follows from Lemma \ref{Lem5.7} that
\begin{align}\label{h83}
\mathcal{I}_q(w_n) \geq \beta_0 | \nabla w_n|_2^2+o (1).
\end{align}
Moreover, since $u\in V_a$,  by \eqref{h81} one gets
\begin{align*}
m(a)+o (1)=\mathcal{I}_q(u_n)=\mathcal{I}_q(w_n)+\mathcal{I}_q(u)+o (1)\geq \mathcal{I}_q(w_n)+m(a)+o (1),
\end{align*}
which and \eqref{h83} show that $|\nabla w_n|_2\rightarrow 0$   as $n\rightarrow\infty$. Hence, we have $ w_n \rightarrow 0$ in $E$  as $n\rightarrow\infty$. That is, $u_n\rightarrow u$ in $E$  as $n\rightarrow\infty$.  Furthermore, we see from   \eqref{h81}, \eqref{h88} and \eqref{h89}   that
  $ \mathcal{I}_q(u) = m(a ) $. Thus, we complete the proof.
  \end{proof}

{\bf {Proof of Theorem \ref{TH6}}}: Let $\{u_n\}\subset D_{k_0}$ and $|u_n|_2\rightarrow a$
be a minimizing sequence for $m(a)$ with $a\in(0, a_0)$. Then $\{u_n\}$ is bounded in $E$. Hence, there exists $u\in E$ such that, up to a subsequence,
\begin{align*}
&u_n\rightharpoonup u \quad\quad\quad\quad \ \ \text{in}\ E; \\
&u_n\rightarrow u \quad\quad \quad\quad\ \ \text{in}\ L_{loc}^p(\mathbb{R}^N)\ \text{with} \ p\in[1, 2^*);\\
&u_n(x)\rightarrow u(x) \quad\quad  a.e. \ \text{on}\ \mathbb{R}^N.
\end{align*}
Now we show that $u\neq 0$. Assume  by contradiction  that $u=0$. Then
\begin{align*}
m(a)+o (1)=\mathcal{I}_q(u_n)=\mathcal{I}_q^\infty(u_n)+ \int_{\mathbb{R}^N} V(x) |u_n|^2dx+
o (1)\geq m^\infty(a)+o (1),
\end{align*}
which gives that $m(a)\geq m^\infty(a) $ as $n\rightarrow\infty$.  It contradicts to  Lemma \ref{Lem5.5}. So we have $u\neq 0$.  Therefore, we deduce   from Lemmas \ref{Lem5.3} and \ref{Lem5.6} that $u_n\rightarrow u$ in $E$ as $n\rightarrow\infty$ and $u\notin \partial V_a$ is a minimizer for $\mathcal{I}_q$ on $V_a$, that is, $ \mathcal{I}_q(u)=m(a)$. Then,  by the Lagrange multiplier,
there exists $\lambda_a \in \mathbb{R}$   such that
\begin{align} \label{h84}
  - \Delta u+V(x)u+\lambda_a u=|u|^{2^*-2}u+\mu |u|^{q-2}u \quad \quad \text{in} \ \mathbb{ R}^N.
  \end{align}
If assume that $(V_4)$ holds, similar to \eqref{h85} we know that $\lambda_a>0$. 
Thus we obtain that there exists a couple $(\lambda, u)\in \mathbb{R}^+\times E$ solving Eq. \eqref{h1} with $q\in (2, 2+\frac{4}{N})$, where $ \mathcal{I}_q(u)=m(a)<0$. So we complete the proof of Theorem  \ref{TH6}. $\hfill\Box$\\
\section{ Proof of Theorems \ref{JTH2}, \ref{TH4} and \ref{TH5}}\label{sec5}

In this section, we focus on the properties of  the normalized solution for Eq. \eqref{h1} with $q\in[2+\frac{4}{N}, 2^*)$.   Firstly, we give the exponential decay property of the  positive normalized solution   of Eq. \eqref{h1} and prove  Theorem \ref{JTH2}.
Then we verify that the    solution  of the initial-value problem  \eqref{h51}
with initial datum    blows-up in finite time  in Theorem \ref{TH4} under some suitable assumptions. Finally, we study the strong instability of the standing waves for  problem  \eqref{h51} in  Theorem \ref{TH5}. \\

 {\bf {Proof of Theorem \ref{JTH2}}}: Let $u\in E$ be the positive real-valued solution
for Eq. \eqref{h1} with  $\mu, a>0$  and $q\in[2+\frac{4}{N}, 2^*)$,
 then it follows $(V_4)$ and \eqref{h85} that the corresponding  Lagrange multiplier $\lambda>0$,
  Now,  we claim that the   positive real-valued  normalized     solution $u$ of Eq. \eqref{h1}   decays exponentially. In fact, we first show that $|u(x)| \rightarrow  0$ as $|x|\rightarrow\infty$.    For any $k>0$, define
  \begin{align*} u_k=
   \begin{cases}
   u,   \quad \quad&|u(x)|\leq k,\\
   \pm k,   \quad \quad&\pm u(x)>k.
  \end{cases}
  \end{align*}
Let $\vartheta\in C^\infty(\mathbb{R}^N,[0,1]) $ such that $|\nabla \vartheta|^2\leq\frac{4}{R^2}$ and
  \begin{align*} \vartheta=
   \begin{cases}
   1,  \quad \quad &| x|\geq R,\\
  0,  \quad \quad &| x|\leq r,
  \end{cases}
  \end{align*}
where $R>2$ and $1<r<\frac{R}{2}$.  For all $\iota\geq1$, Multiplying both sides of Eq. \eqref{h1} by $\varphi_k:=\vartheta^2|u_k|^{2\iota}u\in E$ and integrate on $\mathbb{R}^N$,  one obtains
 \begin{align} \label{h62}
\int_{\mathbb{R}^N} \nabla u \cdot \nabla \varphi_k+ \lambda u \varphi_k+V(x) u \varphi_k dx=\int_{\mathbb{R}^N} |u|^{2^*-2} u \varphi_k+\mu|u|^{q-2} u \varphi_kdx.
\end{align}
Since   $\lambda>0$ and
$$
\nabla \varphi_k=2\vartheta^2 \iota |u_k|^{2\iota-2} u_k u \nabla u_k+ \vartheta^2 |u_k|^{2 \iota} \nabla u+2 \vartheta |u_k|^{2\iota} u \nabla \vartheta,
$$
then by \eqref{h62} we get
 \begin{align}\label{h63}
&\int_{\mathbb{R}^N} 2\vartheta^2 \iota |u_k|^{2\iota-2} u_k u \nabla u_k\cdot \nabla udx+  \int_{\mathbb{R}^N} \vartheta^2 |u_k|^{2 \iota} |\nabla u|^2dx\nonumber\\
\leq& \Big|\int_{\mathbb{R}^N} 2 \vartheta |u_k|^{2\iota} u \nabla \vartheta\cdot \nabla u dx\Big|+ \int_{\mathbb{R}^N} |V(x)| u \varphi_k dx+ \int_{\mathbb{R}^N} |u|^{2^*-2} u \varphi_kdx+ \mu\int_{\mathbb{R}^N}|u|^{q-2} u \varphi_kdx\nonumber\\
:= &I_1+I_2+I_3+I_4,
\end{align}
where
\begin{align*}
I_1&=\Big|\int_{\mathbb{R}^N} 2 \vartheta |u_k|^{2\iota} u \nabla \vartheta\cdot \nabla u dx\Big|; \quad\quad \
I_2=\int_{\mathbb{R}^N} |V(x)| u \varphi_k dx;\\
I_3&=\int_{\mathbb{R}^N} |u|^{2^*-2} u \varphi_kdx;\quad\quad\quad\quad\quad\quad
I_4=\mu\int_{\mathbb{R}^N}|u|^{q-2} u \varphi_kdx.
\end{align*}
 For $I_1 $, by Young inequality,
 \begin{align*}
 I_1=\Big|\int_{\mathbb{R}^N} 2 \vartheta |u_k|^{2\iota} u \nabla \vartheta\cdot \nabla u dx\Big|
 &=\Big|\int_{\mathbb{R}^N} \vartheta|u_k|^{\iota} \nabla u \cdot 2|u_k|^{\iota} u \nabla \vartheta dx\Big|\\
  &\leq \frac{1}{2} \int_{\mathbb{R}^N} \vartheta^2|u_k|^{2\iota} |\nabla u|^2dx + \int_{\mathbb{R}^N}|u_k|^{2\iota} u^2 |\nabla \vartheta|^2 dx\\
  &\leq \frac{1}{2} \int_{\mathbb{R}^N} \vartheta^2|u_k|^{2\iota} |\nabla u|^2dx+ C \int_{|x|\geq r}|u_k|^{2\iota} u^2 dx.
 \end{align*}
  For $I_2 $, since $V\in C(\mathbb{R}^N, \mathbb{R})$ and  $\lim_{|x|\rightarrow\infty} V(x)=0$, one gets
 $$
 I_2=\int_{\mathbb{R}^N} |V(x)| u \varphi_k dx\leq C\int_{|x|\geq r}|u_k|^{2\iota} u^2 dx.
 $$
  For $I_3 $, by Lemma \ref{Lem4.3},
$$
I_3=\int_{\mathbb{R}^N} |u|^{2^*-2} u \varphi_kdx\leq C\int_{|x|\geq r}|u_k|^{2\iota} u^2 dx.
$$
Similarly,
$$
I_4=\mu\int_{\mathbb{R}^N}|u|^{q-2} u \varphi_kdx\leq C(\mu,q)\int_{|x|\geq r}|u_k|^{2\iota} u^2 dx.
$$
Then,  from \eqref{h63}, there exists positive constant $C$, depending on $\mu, u, q$, such  that
\begin{align}\label{h64}
 \int_{\mathbb{R}^N} 2\vartheta^2 \iota |u_k|^{2\iota-2} u_k u \nabla u_k\cdot \nabla udx+   \int_{\mathbb{R}^N} \vartheta^2 |u_k|^{2 \iota} |\nabla u|^2dx
\leq C\int_{|x|\geq r}|u_k|^{2\iota} u^2 dx.
\end{align}
Moreover, for all $\iota\geq 1$,
\begin{align}\label{h65}
& \Big(\int_{|x|\geq R}(|u_k|^{\iota} u)^{2^*} dx\Big)^{\frac{2}{2^*}} \nonumber\\
 \leq& \Big(\int_{\mathbb{R}^N} (\vartheta |u_k|^{\iota} u) ^{2^*} dx\Big)^{\frac{2}{2^*}} \nonumber\\
 \leq&   \int_{\mathbb{R}^N} |\nabla (\vartheta |u_k|^{\iota} u)|^{2} dx \nonumber\\
 \leq&    \int_{\mathbb{R}^N} ( |u_k|^{\iota} u \nabla \vartheta+ \iota \vartheta u |u_k|^{\iota-2} u_k \nabla u_k+ \vartheta |u_k|^\iota \nabla u)^2dx\nonumber\\
 \leq& \int_{\mathbb{R}^N} |u_k|^{2\iota} u^2 |\nabla \vartheta|^2+ \iota^2 \vartheta^2 u^2 |u_k|^{2\iota-2} | \nabla u_k|^2+ \vartheta^2 |u_k|^{2\iota} |\nabla u|^2dx\nonumber\\
 \leq& C  \int_{|x|\geq r} |u_k|^{2\iota} u^2   dx+ \iota\Big(2 \int_{\mathbb{R}^N}\iota \vartheta^2 u^2 |u_k|^{2\iota-2} | \nabla u_k|^2dx+ \int_{\mathbb{R}^N} \vartheta^2 |u_k|^{2\iota} |\nabla u|^2dx\Big)\nonumber\\
  \leq& C(1+ \iota)\int_{|x|\geq r}|u_k|^{2\iota} u^2dx,
\end{align}
where   \eqref{h64} and the definition of $u_k $  are applied to
the last inequality.   Let $k\rightarrow\infty$,  \eqref{h65} implies that
\begin{align}\label{h66}
|u|_{L^{2^*(\iota+1)}(|x|\geq R)} \leq [C(1+ \iota)]^{ \frac{1}{2(1+\iota)}}|u|_{L^{2(\iota+1)}(|x|\geq r)}.
\end{align}
In order to use the Moser iteration, let
$$
\iota_1=\frac{2^*}{2}-1>0,   \ \ \ \ (1+\iota_{n+1})=\frac{2^*}{2}(1+\iota_{n}), \ \ \  \
R_{n}=r_{n+1}=R-(R-r)\left(\frac{R-r}{R}\right)^{n},
$$
 then, by \eqref{h66},
\begin{align}\label{h67}
|u|_{L^{2^*(1+\iota_{n+1})}(|x|\geq R_{n+1})}
\leq& \left[C(1+\iota_{n+1})\right]^{\frac{1}{2(1+\iota_{n+1})}}|u|_{L^{2(1+\iota_{n+1})}(|x|\geq r_{n+1})}\nonumber\\
=& \left[C(1+\iota_{n+1})\right]^{\frac{1}{2(1+\iota_{n+1})}}|u|_{L^{2^*(1+\iota_{n})}(|x|\geq R_{n})}\nonumber\\
\leq& \left[C(1+\iota_{n+1})\right]^{\frac{1}{2(1+\iota_{n+1})}}
\left[C(1+\iota_{n})\right]^{\frac{1}{2(1+\iota_{n})}}
|u|_{L^{2(1+\iota_{n})}(|x|\geq r_{n})}\nonumber\\
&  \cdots \nonumber\\
\leq& \left[C\right]^{\sum^{n+1}_{i=1}\frac{1}{2(1+\iota_{i})}}
\prod^{n+1}_{i=1}(1+\iota_{i})^{\frac{1}{2(1+\iota_{i})}}
|u|_{L^{2^*}(|x|\geq r_1)}.
\end{align}
Since $\iota_1=\frac{2^*}{2}-1$ and $1+\iota_{n+1}=\frac{2^*}{2}(1+\iota_{n})$, one can see that $1+\iota_n=(\frac{2^*}{2})^{n}$ for all $n\in \mathbb{N}$. Then,
\begin{align*}
\left[C\right]^{\sum^{\infty}_{i=1}\frac{1}{2(1+\iota_{i})}}=
\left[C\right]^{\frac{1}{2}\sum^{\infty}_{i=1}(\frac{2}{2^{*}})^{i}}<+\infty
\end{align*}
and
\begin{align*}
\prod^{\infty}_{i=1}(1+\iota_{i})^{\frac{1}{2(1+\iota_{i})}}
=\prod^{\infty}_{i=1}\Big[\big(\frac{2^*}{2}\big)^{i}\Big]^{ \frac{1}{2}(\frac{2}{2^*})^i }=
\Big(\frac{2^*}{2}\Big)^{\frac{1}{2} \sum^{\infty}_{i=1} i( \frac{2}{2^*})^i}<+\infty.
\end{align*}
Therefore,   letting $n\to\infty$ in \eqref{h67}, we can conclude that
\begin{align*}
|u|_{L^{\infty}(|x|\geq R)}
\leq C
|u|_{L^{2^*}(|x|\geq r)}.
\end{align*}
Hence, for any $\epsilon>0$ fixed,  choosing $r > 1$ large enough one infers   $|u|_{L^{\infty}(|x|\geq R)}\leq \epsilon$. This shows that
\begin{align}\label{h68}
|u(x)| \to 0\ \ \ \ \  \mathrm{as }~|x|\to \infty.
\end{align}
Next, for any $q\in[2+\frac{4}{N}, 2^*)$ and $\mu>0$, by \eqref{h68}  and the fact that  $\lim_{|x|\rightarrow\infty} V(x)=0$, there exists    $\widetilde{R}>0$ large enough such that
 \begin{align*}
-\Delta u=(-V(x)-\lambda+|u|^{2^*-2}+\mu|u|^{q-2} )u\leq \frac{1}{2} u     \quad \quad \text{for all}\ |x|>\widetilde{R}.
 \end{align*}
Let $\xi(x)=M_1e^{\big(-\sqrt{\frac{1}{2}}|x|\big)}$, where $M_1$ satisfies
 \begin{align*}
M_1e^{-\sqrt{\frac{1}{2}}\widetilde{R}}\geq u(x)\quad\quad \text{for}\ |x|=\widetilde{R}.
 \end{align*}
By simple calculation,  we get $ \Delta \xi \leq\frac{1}{2} \xi$ for all $x\neq0$. Set $\zeta=\xi-u$, one has
 \begin{align*}
 \begin{cases}
-\Delta \zeta+\frac{1}{2} \zeta \geq 0,\quad &|x|>\widetilde{R},\\
\zeta(x)\geq0,\quad &|x|=\widetilde{R},\\
\lim_{|x|\rightarrow\infty}\zeta(x)=0.
 \end{cases}
 \end{align*}
Hence, it follows from maximum principle \cite{2001book} that $\zeta\geq0$. That is,
$$
|u(x)|\leq M_1e^{\big(-\sqrt{\frac{1}{2}}|x|\big)}\quad\quad \text{ for all} \ |x|\geq \widetilde{R}.
$$
Moreover, since $u$ is continuous function, there exists $M_2>0$ such that
$$
|u(x)| e^{\big(-\sqrt{\frac{1}{2}}|x|\big)} \leq M_2 \quad\quad \text{ for all} \ |x|\leq\widetilde{R}.
$$
Therefore, choosing $M=\max\{ M_1, M_2\}$, we have
$$
|u(x)|\leq Me^{\big(-\sqrt{\frac{1}{2}}|x|\big)}\quad\quad \text{ for all} \ x\in \mathbb{R}^N.
$$
Thus, we complete the proof of Theorem \ref{JTH2}. $\hfill\Box$\\


 Now, we show that  the phenomenon of finite-time blow-up  occurs  for the solution of the  initial-value problem \eqref{h51} with initial datum
  under the assumptions  of   Theorem \ref{TH4}.\\

 {\bf {Proof of Theorem \ref{TH4}}}:   {\bf{(i)}} Let $q\in[\overline{q}, 2^*)$ and let $u_0 \in S_a$ be such that $\mathcal{I}_q(u_0)<\inf_{u\in \mathcal{P}_{q,a}}\mathcal{I}_q(u)$.
 Suppose that  $(V_5)$ holds, according to \cite[Section 3]{2007CPDETao} and \cite{2003-C, 2022-JMPA-jean},  the  initial-value problem \eqref{h51} with initial datum $u_0$  is local well-posed on $(-T_{min}, T_{max})$ with $T_{min}, T_{max}>0 $.
    Furthermore,  the solution to  problem \eqref{h51} with initial datum  $u_0$  has  conservation
of mass and   energy. That is, let $ \phi(t,x)$  be  the solution of  the  initial-value problem \eqref{h51}
with initial datum $u_0$ on $(-T_{min}, T_{max})$,  it holds that
\begin{align}\label{h53}
\|u_0\|_2=\| \phi\|_2\quad\quad \text{and} \quad \quad \mathcal{I}_q(u_0)=\mathcal{I}_q( \phi).
\end{align}
  By the fact that  $|x|u_0 \in L^2(\mathbb{R}^N ,\mathbb{C})$  and \cite[Proposition 6.5.1]{2003-C}, we get
\begin{align} \label{h55}
H(t):=\int_{\mathbb{R}^N}|x|^2|\phi(t,x)|^2dx<+\infty  \ \ \ \ \ \text{for all}\ t\in(-T_{min}, T_{max} ).
  \end{align}
Moreover, the function $H\in C^2(-T_{min}, T_{max} )$ and the following  Virial identity holds:
\begin{align}\label{h52}
H'(t)= -4y(t)  \quad \quad \text{and} \quad \quad H''(t)=-4 y'(t)= 8 P_q(\phi),
  \end{align}
  where
  \begin{align}
  y(t)&= -Im  \int_{\mathbb{R}^N}\overline{\phi} (x\cdot \nabla\phi )dx;\label{h69}\\
  y'(t)&=  -2\int_{\mathbb{R}^N}|\nabla \phi|^2dx+\int_{\mathbb{R}^N} \langle \nabla V(x), x\rangle |\phi|^2dx+2 \int_{\mathbb{R}^N}|\phi|^{2^*}dx+2\mu \gamma_q \int_{\mathbb{R}^N}|\phi|^{q}dx. \label{h70}
  \end{align}
From Lemma \ref{Lem2.3} or  Lemma \ref{Lem2.3j}, for any $u\in S_a$, the function $\varphi_u(s):=\psi_u(\log s)$ with $s>0$ has a   unique global maximum point $\widehat{s}_u=e^{s_u}$ and $\varphi_u(s)$ is strictly decreasing and concave on $(\widehat{s}_u, +\infty)$. According to the assumption $s_{u_0}<0$,   one obtains $ \widehat{s}_{u_0}\in(0,1)$. We claim that
\begin{align}\label{h54}
\text{if}\ u \in S_a \ \text{and}\ \widehat{s}_u\in(0,1), \ \text{then}\ P_q(u)\leq \mathcal{I}_q(u)-\inf_{\mathcal{P}_{q,a}}\mathcal{I}_q.
\end{align}
In fact, since $ \widehat{s}_u\in(0,1)$ and $\varphi_u(s)$ is strictly decreasing and concave on $(\widehat{s}_u, +\infty)$, we infer that $P_q(u)<0$ and
\begin{align*}
\mathcal{I}_q(u)=\psi_u(0) = \varphi_u(1)\geq  \varphi_u(\widehat{s}_u )-\varphi_u'(1)(\widehat{s}_u-1 )
=\mathcal{I}_q(s_u \ast u) -|P_q(u)|(1-\widehat{s}_u )
\geq \inf_{\mathcal{P}_{q,a}}\mathcal{I}_q +P_q(u),
\end{align*}
which completes the claim.
Now, let us consider the solution $\phi$  for the  initial-value problem \eqref{h51} with initial datum $u_0$. Since by assumption $s_{u_0} < 0$, and the map $u \mapsto s_u$ is continuous, we deduce that  $s_{ \phi(t)} < 0$   for every $|t|<\overline{t}$ with $\overline{t}>0$ small enough. Then $\widehat{s}_{ \phi(t)}\in (0,1)$ for $|t|<\overline{t}$. By \eqref{h54}  and recalling the assumption     $\mathcal{I}_q(u_0)< \inf_{\mathcal{P}_{q,a}}\mathcal{I}_q $, we deduce from \eqref{h53} that
$$
P_q(\phi(t))\leq \mathcal{I}_q(\phi(t) )-\inf_{\mathcal{P}_{q,a}}\mathcal{I}_q=\mathcal{I}_q(u_0 )-\inf_{\mathcal{P}_{q,a}}\mathcal{I}_q:=-\eta<0,
$$
for every such $|t|< \overline{t}$. Next, we show that
\begin{align}\label{h56}
  P_q(\phi(t) )\leq-\eta \quad\quad \text{ for any }\ t\in (-T_{min}, T_{max}).
\end{align}
 Assume that there is $t_0\in (-T_{min}, T_{max})$ satisfying $ P_q(\phi(t_0) )=0$. It holds from \eqref{h53} that $\phi(t_0)\in S_a$  and
$$
 \mathcal{I}_q(\phi(t_0) )\geq m_{q,a}=\inf_{\mathcal{P}_{q,a}}\mathcal{I}_q> \mathcal{I}_q(u_0)=  \mathcal{I}_q(\phi(t_0) ),
$$
which is a contradiction.  This shows that    $ P_q(\phi(t_0) )\neq0$ for any  $t_0\in (-T_{min}, T_{max})$. Then, since $P_q(\phi(t))<0$ for every   $|t|< \overline{t}$, we obtain $ P_q(\phi(t) )<0$ for any  $t\in (-T_{min}, T_{max})$ (if at some $t\in  (-T_{min}, T_{max})$, $P_q(\phi(t))>0$. By continuity, we have $P_q(\phi(\widetilde{t}))=0$ for some $\widetilde{t}\in  (-T_{min}, T_{max})$, a contradiction). Since $ P_q(\phi(t) )<0$ for any  $t\in (-T_{min}, T_{max})$, by Lemma \ref{Lem2.3}, $s_{\phi(t) }<0$. That is, $\widehat{ s}_{\phi(t) }\in (0,1)$. Therefore, \eqref{h54} holds and   the above arguments yield
\begin{align*}
  P_q(\phi(t) )\leq-\eta \quad\quad \text{ for any }\ t\in (-T_{min}, T_{max}).
\end{align*}
Thus, we deduce from \eqref{h55}, \eqref{h52} and \eqref{h56} that
\begin{align*}
0\leq H(t)\leq H(0)+H'(0) t+\frac{1}{2}H''(0) t^2\leq H(0)+H'(0) t-4 \eta t^2 \quad\quad \text{ for any }\ t\in (-T_{min}, T_{max}),
\end{align*}
which implies that
 $T_{max}$ has   an upper bound since the right hand side becomes negative for $t$ large.
 Hence, there exists $T\in (0, T_{max})$ such that $\lim_{t\rightarrow T}H(t)=0$.  Since
 $$
 \int_{\mathbb{R}^N} |\phi |^2dx\leq C \left(\int_{\mathbb{R}^N}|x|^2 |\phi|^2dx\right)^{\frac{1}{2}}  \left(\int_{\mathbb{R}^N} \frac{1}{|x|^2}| \phi|^2dx\right)^{\frac{1}{2}}\leq    C    \left(\int_{\mathbb{R}^N}|x|^2 |\phi|^2dx\right)^{\frac{1}{2}}  \left(\int_{\mathbb{R}^N} |\nabla \phi|^2dx\right)^{\frac{1}{2}},
 $$
 where the H\"{o}lder  and Hardy inequality are applied,  we get from   \eqref{h53} that $| \nabla  \phi|_2 \rightarrow +\infty$ as $t \rightarrow T$.
 That is,
the    solution $\phi$  of the initial-value problem  \eqref{h51}
with initial datum $u_0$   blows-up in finite time.

{\bf{(ii)}}  Let $q\in(\overline{q}, 2^*)$.
  Suppose      that  $\mathcal{I}_q(u_0)<0$,  $|x|u_0 \in L^2(\mathbb{R}^N, \mathbb{C})$  and $y(0)=y_0>0$  defined in \eqref{h69}.  We know that   \eqref{h55}--\eqref{h70} still hold.
  To prove  the phenomenon of blowup under the above assumptions,
   we will follow the convexity method of Glassey \cite{1977-Glassey}.  More precisely,
we   show  that  the variance  function $H(t)$ defined in \eqref{h55} is decreasing and concave  for $t > 0$, which suggests
the existence of a blowup time $T$ at which  $H(T)= 0$.
Moreover,  we  will give the estimate  of the blow-up time $T$.  In view of $(V_1)$, $(V_2)$, \eqref{h53}, \eqref{h70} and the fact that $\mathcal{I}_q(u_0)<0$, one infers
\begin{align}\label{h71}
y'(t)= & -2\int_{\mathbb{R}^N}|\nabla \phi|^2dx+\int_{\mathbb{R}^N} \langle \nabla V(x), x\rangle |\phi|^2dx+2 \int_{\mathbb{R}^N}|\phi|^{2^*}dx+2\mu \gamma_q \int_{\mathbb{R}^N}|\phi|^{q}dx\nonumber\\
 =&(q\gamma_q-2) \int_{\mathbb{R}^N}|\nabla \phi|^2dx+ \int_{\mathbb{R}^N} \langle \nabla V(x), x\rangle |\phi|^2dx+
q\gamma_q \int_{\mathbb{R}^N}   V(x) |\phi|^2dx\nonumber\\
&+\big(2-\frac{2q\gamma_q }{2^*}\big)\int_{\mathbb{R}^N}|\phi|^{2^*}dx-2q\gamma_q \mathcal{I}_q(\phi )\nonumber\\
\geq& (q\gamma_q-2-q\gamma_q\sigma_1-\frac{\sigma_2}{2} )\int_{\mathbb{R}^N}|\nabla \phi|^2dx,
\end{align}
where $q\gamma_q-2-q\gamma_q\sigma_1-\frac{\sigma_2}{2} >0 $. Hence, we have
$$
y'(t)>0\quad  \text{for all} \ t\in(-T_{min}, T_{max}) \quad\text{and}\quad  y(t)>y(0)=y_0>0 \quad \text{for all} \ t>0,
$$
which shows that $ H'(t)<0$ for   $t>0$  and $H''(t)<0$  for    $t\in(-T_{min}, T_{max})$.  Then we obtain the function $ H$ is decreasing for any $t > 0$ and  the function $ H$ is concave  for all   $t\in(-T_{min}, T_{max})$.
This  shows the existence of a blowup time  $T$ at which $H(T)= 0$.
Moreover, by H\"{o}lder inequality, for all   $t\in(-T_{min}, T_{max})$, we  deduce
$$
y(t)\leq |x \phi|_2|\nabla \phi|_2,
$$
which and  the monotonicity  and concavity of the function $H(t)$ imply that
\begin{align}\label{h72}
|\nabla \phi|_2\geq  \frac{y(t)}{|x \phi|_2}\geq   \frac{y(t)}{|x u_0|_2}.
\end{align}
In what follows,  combining \eqref{h71} with \eqref{h72}, we get the following  ODE for $y$,
\begin{align*}
\begin{cases}
y'(t)\geq ( q\gamma_q-2-\frac{1}{2}\sigma_2-q\gamma_q \sigma_1 ) \frac{|y(t)|^2}{|x u_0|_2^2},\\
y(0)=y_0>0,
\end{cases}
\end{align*}
which implies that there exists
$$
0<T\leq\frac{|xu_0|_2^2}{( q\gamma_q-2-\frac{1}{2}\sigma_2-q\gamma_q \sigma_1 ) y_0}
$$
satisfying $\lim_{t\rightarrow T^-}|\nabla \phi|=+\infty $. Thus,   we complete the proof
of Theorem \ref{TH4}. $\hfill\Box$ \\

In the following, we are going to prove   the strong instability of the standing wave for problem \eqref{h51} by  using   Theorems \ref{JTH2} and \ref{TH4}.\\

 {\bf {Proof of Theorem \ref{TH5}}}:  Under the assumptions of Theorem \ref{TH1} or Theorem \ref{TH2},
 we   first describe the characteristic of $Z_a$  as
\begin{align}\label{h94}
Z_a=\big\{ e^{i \theta}u:  \ \theta\in \mathbb{R}, \ u\in  U_a\ \text{and}\   u >0 \ \text{in}\ \mathbb{R}^N \big\},
\end{align}
where $U_a$ is defined in \eqref{h91}.
For any $z\in E$,   let $z(x)=(v(x),w(x))= v(x)+iw(x)$, where $  v, w\in E$ are real-valued functions   and
$$
\|z\|^2=|\nabla z|_2^2+| z|_2^2, \quad\quad | z|_2^2=|v|_2^2+| w|_2^2\quad \quad \text{and}\quad \quad  | \nabla z|_2^2=|\nabla v|_2^2+|\nabla w|_2^2.
$$
Taking $z=(v,w)\in Z_a$,   we see from \eqref{h90} that $ |z|\in  Z_a$ and $ | \nabla |z||_2^2=| \nabla z|_2^2$. Then, by the fact that $ | \nabla |z||_2^2-| \nabla z|_2^2=0$ one obtains
\begin{align}\label{h92}
\int_{\mathbb{R}^N}\sum_{i=1}^N \frac{(v \partial_i w-w\partial_i v)^2}{v^2+w^2}dx =0.
\end{align}
Hence, from \cite[Theorem 4.1]{2004-ANS-Ha}, we know that
 \begin{itemize}
   \item [$(i)$] either $v\equiv 0$ or $v(x) \neq   0$ for all $x \in \mathbb{R}^N$;
   \item [$(ii)$] either $w\equiv 0$ or $w(x) \neq   0$ for all $x \in \mathbb{R}^N$.
 \end{itemize}
 Now we turn to the characterization of $Z_a$.
On the one hand, let $v= e^{i \theta}u$ with  $\theta\in \mathbb{R}$,  $  u\in  U_a$   and $ u>0$, then we have $|v|_2^2=a^2$ and
$$
\mathcal{I}_q(v)=\mathcal{I}_q(u)= m_{q,a}   \ \ \ \text{and}  \ \ \ P_q(v)=P_q(u)=0,
$$
 which implies that
 $$ \big\{ e^{i \theta}u:  \ \theta\in \mathbb{R}, \ u\in  U_a\ \text{and}\   u >0 \ \text{in}\ \mathbb{R}^N \big\}\subseteq U_a=Z_a.
 $$
On the other hand,  let $z=(v,w)\in Z_a$, we have  $u:= |z|\in  Z_a$.   If $w\equiv 0$, then   we deduce that $u:= |z|=|v|>0$ on $\mathbb{R}^N$  and $z=   e^{i \theta} u $ where  $\theta =0$ if $v>0$ and   $\theta =\pi$ if $v<0$ on $\mathbb{R}^N$.  Otherwise, it follows from $(ii)$  that $w(x)\neq 0$ for all $x\in\mathbb{R}^N$. Since
\begin{align*}
  \frac{(v \partial_i w-w\partial_i v)^2}{v^2+w^2}=\Big[\partial_i\big (\frac{v}{w}\big)  \Big ]^2 \frac{w^2}{v^2+w^2}\quad\quad \text{where}\ i=1,2, \ldots, N,
\end{align*}
for all $x\in \mathbb{R}^N$,  in view of \eqref{h92} we get
$$
\nabla \Big(\frac{v}{w}\Big)=0 \quad\quad \text{ on }\  \mathbb{R}^N.
$$
Therefore, there exists $ C\in \mathbb{R}$ such that $v=C w$ on  $\mathbb{R}^N$. Then we have
\begin{align}\label{h93}
z=(v,w)=v+iw=(C+i)w \quad\quad\text{and}\quad\quad |z|= |C+i||w|.
\end{align}
Let $\theta_1\in R$ be such that $C+i= |C+i| e^{i \theta_1}$ and let $w=|w| e^{i \theta_2} $ with
\begin{align*} \theta_2=
\begin{cases}
0, \quad\quad &\text{if}\ w>0;\\
\pi, \quad\quad &\text{if}\ w<0.
\end{cases}
\end{align*}
 Then we can see from \eqref{h93} that $ z=(C+i)w=|C+i| |w|e^{i (\theta_1+ \theta_2)} = |z|e^{i (\theta_1+ \theta_2)}$. Setting $\theta=\theta_1+\theta_2$ and $u:=|z|$, then $0<u\in U_a$ and $z= e^{i \theta }u$.
Thus,
$$
Z_a= \big\{ e^{i \theta}u:  \ \theta\in \mathbb{R}, \ u\in  U_a\ \text{and}\  u>0 \ \text{on}\ \mathbb{R}^N \big\}.
$$
Thus, we know that  if $\widehat{u} \in Z_a$,   we have $\widehat{u}=e^{i \theta}u $ for $\theta\in \mathbb{R}$ and $0<u\in  U_a$, and then, similar to \eqref{h85},    the associated Lagrange multiplier $\widehat{\lambda}>0$. 

  {\bf  {  Strong instability of the stand wave $e^{i
 \widehat{\lambda} t} \widehat{u}$.  }}
  For any $q\in [\overline{q}, 2^*)$, let    $\phi_\varrho(x, t)$ be the solution to   the  initial-value problem \eqref{h51} with initial datum $\widehat{u}_\varrho$, where $\widehat{u}_\varrho=\varrho\ast \widehat{u}$ with $\varrho>0$.  We have $\widehat{u}_\varrho \rightarrow \widehat{u} $ in $E$ as $\varrho\rightarrow 0^+$, and hence it is sufficient to prove
that $\phi_\varrho(x, t)$  blows-up in finite time. Let $s_{\widehat{u}_\varrho}\in \mathbb{R}$ be defined in Lemma \ref{Lem2.3} or Lemma \ref{Lem2.3j}. Clearly $s_{\widehat{u}_\varrho} = -\varrho < 0$, and
$$
 \mathcal{I}_q(\widehat{u}_\varrho)= \mathcal{I}_q(\varrho\ast \widehat{u})<\mathcal{I}_q(  \widehat{u})=\inf_{\mathcal{P}_{q,a}}\mathcal{I}_q.
$$
Furthermore, from \eqref{h90} and \eqref{h94}  we have $0<|\widehat{u}|\in Z_a $.  Since  $|\widehat{u}|  $ decays exponentially according to Theorem \ref{JTH2},  we deduce
$$
\int_{\mathbb{R}^N} |x|^2|\widehat{u}_\rho|^2dx=e^{-2 \rho} \int_{\mathbb{R}^N} |x|^2|\widehat{u}|^2dx<+\infty,
$$
which shows tht
$ |x|\widehat{u}_\varrho\in L^2(\mathbb{R}^N, \mathbb{C})$.
Hence, it follows from Theorem \ref{TH4}-(i)  that the solution $\phi_\varrho(x, t)$  blows-up in finite time. Thus, by  Definition \ref{de1.2}, the
standing wave $e^{i\lambda t}\widehat{u} $ is strongly unstable. We complete the proof of Theorem \ref{TH5}.$\hfill\Box$\\

\hspace{-0.90cm}{\bf{Acknowledgments}}

  This work  was   supported by  National Natural Science Foundation of China (No. 11971393).\\


\end{document}